\documentclass[12pt]{article}

\RequirePackage{kpfonts}
\RequirePackage[sf,bf,small,raggedright]{titlesec}
\usepackage{textcomp}
\usepackage{amssymb}
\usepackage{bbm}
\usepackage{fancyhdr}
\usepackage{amsmath}
\usepackage{amsbsy,amsthm}
\usepackage{amscd}
\usepackage{latexsym}
\usepackage{graphicx}   % to include and manipulate graphics
\usepackage{pdfsync}
\usepackage{blkarray}
\usepackage{multirow}
\usepackage{hyperref}
\usepackage{color}
\usepackage{comment}
\usepackage{framed}

\newcommand{\R}{\mathbb{R}}

\newcommand{\inr}[1]{\left\langle #1 \right\rangle}
\newcommand{\ind}{\mathbbm{1}}

\newcommand{\E}{\mathbb{E}}

\newcommand{\eps}{\varepsilon}

\newcommand{\cF}{{\cal F}}

\newtheorem{lemma}{Lemma}
\newtheorem{theorem}{Theorem}

\numberwithin{equation}{section}

\def \remark {\noindent {\bf Remark.}\ \ }

\def\IND{\mathbbm{1}}

\newcommand{\X}{\mathcal{X}}

\newcommand{\cE}{\mathcal{E}}

\newcommand{\EXP}{\mathbb{E}}
\renewcommand{\E}{\mathbb{E}}
\newcommand{\PROB}{\mathbb{P}}

\newcommand{\var}{\mathrm{Var}}

\begin{document}

\title{Noise sensitivity of the top eigenvector of a Wigner matrix
\thanks{
G\'abor Lugosi was supported by
the Spanish Ministry of Economy and Competitiveness,
Grant PGC2018-101643-B-I00;
``High-dimensional problems in structured probabilistic models - Ayudas Fundaci\'on BBVA a Equipos de Investigaci\'on Cientifica 2017'';
and Google Focused Award ``Algorithms and Learning for AI''. Charles Bordenave was supported by by the research grants ANR-14-CE25-0014 and ANR-16-CE40-0024-01. Nikita Zhivotovskiy was supported by RSF grant No. 18-11-00132.
}}
\author{Charles Bordenave\thanks{Institut de Math\'ematiques de Marseille, CNRS \& Aix-Marseille University, Marseille, France.}  \and G\'abor Lugosi
\thanks{Department of Economics and Business, Pompeu
  Fabra University, Barcelona, Spain, gabor.lugosi@upf.edu}
\thanks{ICREA, Pg. Lluís Companys 23, 08010 Barcelona, Spain}
\thanks{Barcelona Graduate School of Economics}
\and Nikita Zhivotovskiy\thanks{This work was prepared while Nikita Zhivotovskiy was a postdoctoral fellow at the department of Mathematics, Technion I.I.T. and researcher at National University Higher School of Economics. Now at Google Research, Brain Team.}
}

\maketitle

\begin{abstract}
We investigate the noise sensitivity of the top eigenvector of a
Wigner matrix in the following sense. Let $v$ be the top eigenvector 
of an $N\times N$ Wigner matrix. Suppose that $k$ randomly chosen entries of the
matrix are resampled, resulting in another realization of the Wigner
matrix with top eigenvector $v^{[k]}$. We prove that, with high
probability,  when   $k \ll N^{5/3-o(1)}$, then $v$ and $v^{[k]}$ 
are almost collinear and when $k\gg N^{5/3}$, then $v^{[k]}$ is almost orthogonal to
$v$.
\end{abstract}

\section{Introduction}

In this paper we study the \emph{noise sensitivity} of top
eigenvectors of  Wigner matrices. For a positive integer
$N$, let $X=(X_{i,j})$ be a symmetric $N\times N$ matrix such that, for 
$i\leq j$, the $X_{i,j}$ are independent real random
variables, such that for some constant $\delta>0$ and for all $ i \leq
j$, $\E X_{i,j} = 0$ and $\E \exp ( |X_{i,j}|^\delta ) \leq
1/\delta$.   Note that this assumption is satisfied for a wide
  class of distributions with a sufficiently light tail. Uniformly
  bounded, sub-gaussian, and
  sub-exponential distributions fall in this class. 
To guarantee that $X$ is a symmetric matrix, we set
$X_{i,j} = X_{j,i}$. Finally, we assume that the off-diagonal entries
have the unit variance: for all $i \ne j$,  $\E X_{ij}^2 = 1$ and for
all $i$, $\E X_{ii}^2 = \sigma_0^2$, for some $\sigma_0 \geq
0$.
Throughout this text, we 
call such matrix $X$ a Wigner matrix. 
%G 
 In this paper we are concerned with large matrices and
  the main results are asymptotic, concerning $N\to \infty$. 
The distribution of the entries $X_{i,j}$ may change with
$N$ though we suppress this dependence in the notation. However,
the values of $\sigma_0$ and $\delta$ are assumed to be the same for all $N$.

Let $\lambda=\sup_{w\in S^{N-1}} \inr{w,Xw}$ be the top eigenvalue 
of $X$ and let $v$ denote the corresponding   unit eigenvector. In this paper 
we study the noise sensitivity of $v$. In particular, we are
interested in the behavior of the top eigenvector $v^{[k]}$ of the
symmetric matrix $X^{[k]}$
obtained by resampling $k$ random entries of $X$. The main finding of the
paper is that, with high probability, when   $k \leq N^{5/3-o(1)}$, then $v$ and $v^{[k]}$ are almost collinear and when $k\gg N^{5/3}$, then $v^{[k]}$ is almost orthogonal to
$v$.

\subsection*{Related work and proof technique}
 
Noise sensitivity is an important notion in probability that has been
extensively studied since the pioneering work of Benjamini, Kalai, and
Schramm \cite{BeKaSc99}. Noise sensitivity has mostly been studied 
in the context of Boolean functions and it has been shown to have deep 
connections with threshold phenomena, measure concentration, and
isoperimetric inequalities, see
Talagrand \cite{Tal94a},
Friedgut and Kalai \cite{FrKa96},
Kahn, Kalai, and Linial \cite{KaKaLi88},
Bourgain, Kahn, Kalai, Katznelson, and Linial \cite{BoKaKaKaLi92}
for some of the key early work and 
Garban \cite{Gar11},
Garban and Steif \cite{GaSt14},
Kalai and Safra \cite{KaSa06},
O'Donnell \cite{ODo14}
for surveys. 
The key techniques for studying noise sensitivity typically use
elements of harmonic analysis, in particular, hypercontractivity
(\cite{Tal94a}, \cite{KaKaLi88})
but also the ``randomized algorithm'' approach of Schramm and Steif
\cite{ScSt10}
and other techniques, see Garban, Pete, and Schramm \cite{GaPeSc10}.

Our approach is inspired by Chatterjee's work \cite{Cha16} who shows
that, for functions of independent standard Gaussian random variables,
the notion of noise sensitivity (or ``chaos'' as Chatterjee calls it)
is deeply related to the notion of ``superconcentration''.

In fact, a result in a similar spirit to ours for the \emph{Gaussian Unitary Ensemble} was proved by
Chatterjee \cite[Section 3.6]{Cha16}. However, instead of resampling 
random entries of the matrix, the perturbations 
considered in \cite{Cha16} are different. In Chatterjee's model, \emph{every} entry of
the matrix $X$ is perturbed by replacing $X$ by
$Y=e^{-t}X+\sqrt{1-e^{-2t}}X'$ where $X'$ is an independent copy of $X$
and $t>0$. It is proved in \cite{Cha16} that the top eigenvectors of
$X$ and $Y$ are approximately orthogonal (in the sense that the
expectation
of their inner product goes to zero as $N\to\infty$) as soon as $t \gg
N^{-1/3}$.

Chatterjee uses this example to illustrate how ``superconcentration''
implies ``chaos''. His techniques crucially depend on the Gaussian
assumption as in that case explicit formulas may be exploited.
Our techniques are similar in the sense that our starting point is also
``superconcentration'' (i.e., the fact that the variance of the
largest eigenvalue of a Wigner matrix is small). However, outside of the Gaussian realm, 
the notions of superconcentration and chaos are murkier. Starting from
a general formula for the variance of a function of independent random
variables, due to Chatterjee \cite{Cha05}, we establish a monotonicity 
lemma that allows us to make the connection between the variance of
the top eigenvalue and the inner product of interest. Then we use 
the fact that the top eigenvector has a small variance (i.e., in a
sense, it is ``superconcentrated''). The monotonicity lemma 
may be of independent interest and it may have further uses 
when one tries to prove that ``superconcentration implies chaos''
for functions of independent--not necessarily Gaussian--random
variables.

\subsection*{Result}

To formally describe the setup, let $X$ be a symmetric $N\times N$
Wigner matrix as defined above.
For a positive integer $k \le \binom{N}{2} +  N = N(N+1)/2$, let the random matrix $X^{[k]}$ be
defined as follows. Let $S_k=\{(i_1,j_1),\ldots,(i_k,j_k)\}$
be a set of $k$ pairs 
chosen uniformly at random (without replacement) from the set 
%of size $\binom{[N]}{2} + N$ 
of all ordered pairs $(i,j)$ of indices with $1\le
i\le j\le N$. We also assume that $S_k$ is independent of the entries of $X$. The entries of $X^{[k]}$   above the diagonal are
\[
    X^{[k]}_{i,j} = \left\{ \begin{array}{ll}  
                            X'_{i,j} & \text{if $(i,j)\in S_k$}      \\
                             X_{i,j} & \text{otherwise},
                      \end{array} \right.
\]
where $(X'_{i,j})_{1\le i\le j\le N}$ are independent random
variables, independent of $X$ and $X'_{i,j}$ has the same distribution
as $X_{i,j}$, for all $i\le j$. In words, $X^{[k]}$ is obtained from $X$ by resampling $k$
random entries of the matrix   above and including the diagonal and also the corresponding
terms   below the diagonal. Clearly, $X^{[k]}$ has the same distribution
as $X$. Denote unit eigenvectors corresponding to the
largest eigenvalues of $X$ and $X^{[k]}$ by $v$ and $v^{[k]}$,
respectively. 
 Note that with
  overwhelming probability, the spectrum of a Wigner matrix is simple
  and, in particular, the  top unit eigenvector is unique (up to
  changing the sign), see \cite{MR2760897}.

Our main results are the following.

\begin{theorem}
\label{thm:main}
Assume that $X$ is a Wigner matrix as above. If $k/N^{5/3}\to \infty$, then
\[
   \EXP \left|\inr{v,v^{[k]}}\right| = o(1)~.
\]
\end{theorem}

Conversely, our second result asserts that when $k \leq N^{5/3 -o(1)}$ then $v$ and $v^{[k]}$ are almost aligned.
\begin{theorem}
\label{thm:main2}
Assume that $X$ is a Wigner matrix as above. There exists a constant $c >0$ such that, with $\eps_N =( \log N)^{-c \log \log N}$, 
$$
\EXP  \max_{1 \leq k \leq \eps_N N^{5/3} } \min_{s \in \{-1,1\}} \| v - s v^{[k]} \|_{2} = o(1)~.
$$
\end{theorem}
The proof of Theorem \ref{thm:main2} actually establishes that $\max_k \min _s \sqrt N \| v - s v^{[k]} \|_{\infty}$ goes to $0$ in probability. 

The following heuristic argument may provide an intuition of why the
threshold  in the lower bound of Theorem \ref{thm:main2} is at $k = N^{5/3 - o(1)}$. 
Since the seminal work of Erd\H{o}s, Schlein, and Yau \cite{ESY09b},
it is well known that unit eigenvectors of random matrices are
delocalized in the sense that $\| v \|_\infty = N^{-1/2 + o(1)}$ with
high probability. 
  Denoting the top eigenvalue of $X^{[k]}$ by $\lambda^{[k]}$,
we might infer from the derivative of a simple eigenvalue as the function of the matrix entries that
$$
\lambda^{[1]} - \lambda \simeq ( 1+ \IND (i_1 \ne j_1 ))  v_{i_1}( X'_{i_1,j_1} -  X_{i_1,j_1} )v_{j_1} \simeq  \frac{ X'_{i_1,j_1} -  X_{i_1,j_1} }{N^{1+o(1)}}~,
$$
where $v_i$ is the $i$-th component of $v$ 
Assuming that $v_i$ is nearly independent of any matrix entry $X_{ij}$, since $X_{ij}$ is centered with unit variance, we would get from the central limit theorem that
$$
\lambda^{[k]} - \lambda = \sum_{t =0}^{k-1}( \lambda^{[t+1]} - \lambda ^{[t]} )  \simeq \frac{\sqrt k}{N^{1+o(1)}}~ .
$$
On the other hand, the known behavior of random matrices at the edge
of the spectrum implies that the second largest eigenvalue of $X$ is at distance of order $N^{-1/6}$ from $\lambda$. The above heuristic should thus break down when $\sqrt k / N^{1 + o(1)}$ is of order $N^{-1/6}$. It gives the threshold at $k = N^{5/3+o(1)}$.

To get an idea of how Theorem \ref{thm:main} is proved,
consider the variance of the largest eigenvalue $\lambda$ of $X$.
The key inequality we prove is that
\[
 \left(\EXP \left|\inr{v,v^{[k]}}\right|\right)^2 \lesssim \frac{N^2\var(\lambda)}{k}.
\]
By the 
Tracy-Widom law \cite{tracy1994level, tracy1996orthogonal} for the
largest eigenvalue, we expect that $\var(\lambda)$ is of order $N^{-1/3}$, 
which implies the desired asymptotic orthogonality whenever $k/N^{5/3}\to \infty$.
The proof of the inequality above is based on a variance formula for
general functions of independent random variables due to Chatterjee
\cite{Cha05}, see Lemma \ref{lem:chatterjee} below. The variance
formula suggests that small variance implies noise sensitivity of the
top eigenvalue in a certain sense. This is made precise by Lemmas \ref{lem:monotonicity}
and \ref{lem:monotonicity_second}. Finally, noise sensitivity of the
top eigenvalue translates to the inequality above.

\remark
We expect that the arguments of Theorem \ref{thm:main} for the noise sensitivity of the top
  eigenvalue may be modified to prove analogous results for the
  eigenvector corresponding to the $j$-th
largest eigenvalue, $1 \leq j \leq N$. However, the threshold is expected to occur at
values different from $N^{5/3}$. In particular, a simple heuristic
argument suggests that for the $j$-th eigenvector the threshold occurs
around $N^{5/3+o(1)} \min ( j , N -j +1)^{-2/3}$. However, to keep the 
presentation transparent, in this paper we focus on the top eigenvalue.

Interestingly, the proof that the top eigenvalue is very
sensitive to resampling more than $\Theta(N^{5/3})$ entries 
involves proving that it is insensitive to resampling just a single
entry. As a consequence the proofs of Theorems \ref{thm:main} and
\ref{thm:main2} share common techniques.

The rest of the paper is dedicated to proving Theorems \ref{thm:main}
and \ref{thm:main2}. In Section \ref{sec:var} we introduce a general
tool for proving noise sensitivity that generalizes Chatterjee's ideas 
based of ``superconcentration'' to functions of independent, not
necessarily standard normal random variables.
In Section \ref{sec:rm} we summarize some of the tools from random
matrix theory that are crucial for our arguments.
In Sections \ref{sec:thm1} and \ref{sec:thm2} we give the proofs of 
Theorems \ref{thm:main}
and \ref{thm:main2}.

\section{Variance and noise sensitivity} 
\label{sec:var}

The first building block in the proof of Theorem \ref{thm:main} is a formula for the variance of an arbitrary
function of independent random variables, due to Chatterjee
\cite{Cha05}. For any positive integer $i$, denote  $[i]=\{1, \ldots, i\}$. 

\begin{lemma}
\label{lem:chatterjee}
\emph{\cite{Cha05}}
Let $X_1,\ldots,X_n$ be independent random variables taking values in
some set $\X$
and let $f:\X^n \to \R$ be a measurable function.
Denote $X=(X_1,\ldots,X_n)$. Let $X'=(X_1',\ldots,X_n')$ be an independent copy of $X$.
Under the notation
\[
   X^{(i)}=(X_1,\ldots,X_{i-1},X_i',X_{i+1},\ldots,X_n)
\quad \text{and} \quad
   X^{[i]}=(X_1',\ldots,X_i',X_{i+1},\ldots,X_n)
\]
and, in particular, $X^{[0]}=X$ and $X^{[n]}=X'$, we have
\[
\var(f(X)) = \frac{1}{2} \sum_{i=1}^n \EXP \left[\left(f(X)-f(X^{(i)})\right) \left(f(X^{[i-1]})-f(X^{[i]})\right)\right]~.
\]
\end{lemma}

In general, for $A \subseteq [n]$ let $X^A$ denote the random vector, obtained from $X$ by replacing the components indexed by $A$ by corresponding components of $X^{\prime}$.

In the variance formula above, the order of the variables does not
matter and the formula remains valid after permuting the indices
$1,\ldots,n$ arbitrarily. In particular, one may take the variables in 
random order. Thus, if $\sigma=(\sigma(1),\ldots,\sigma(n))$ is a random permutation sampled
uniformly from the symmetric group $S_n$ and $\sigma([i])$ denotes
$\{\sigma(1), \ldots, \sigma(i)\}$,
 then 
\begin{equation}
\label{eq:var}
\var(f(X)) = \frac{1}{2} \sum_{i=1}^n \EXP \left[\left(f(X)-f(X^{{(\sigma(i))}})\right) \left(f(X^{\sigma([i-1])})-f(X^{\sigma([i])})\right)\right]~.
\end{equation}
Note that on the right-hand side of \eqref{eq:var} the expectation is taken
with respect to both $X,X'$, and the random permutation $\sigma$.

One would intuitively expect that the terms on the right-hand side of
\eqref{eq:var} decrease with $i$, as the differences
$f(X)-f(X^{  (\sigma(i))} )$ and $f(X^{\sigma([i-1])})-f(X^{\sigma([i])})$
become less correlated as more randomly chosen components get
resampled. This is indeed the case and this fact is one of our main
tools in proving noise sensitivity. We believe that the following
lemma can be useful in diverse situations. The proof is given in
Section \ref{sec:prooflemmon} below.

\begin{lemma}
\label{lem:monotonicity}
Consider the setup of Lemma \ref{lem:chatterjee} and the notation
above. For $i\in [n]$, denote
\[
   B_i= \EXP \left[\left(f(X)-f(X^{{(\sigma(i))}})\right) \left(f(X^{\sigma([i-1])})-f(X^{\sigma([i])})\right)\right]~,
\]
where the expectation is taken with respect to components of vectors and random permutations. Then $B_i \ge B_{i+1}$ for all $i=1,\ldots,n-1$ and $B_n \ge 0$. In particular, for any $k\in [n]$,
\[
    B_k \le \frac{2\var(f(X))}{k}.
\]
\end{lemma}

We also introduce a modification of Lemma \ref{lem:monotonicity} that
will be more convenient for our purposes. To do so, we introduce the
following notation. Let $j$ have uniform distribution on $[n]$.  Let
$X^{(j) \circ \sigma([i-1])}$ denote the vector obtained  from
$X^{\sigma([i-1])}$ by replacing its $j$-th component by an
independent copy of the random variable $X_j$, denoted by
$X_j^{\prime\prime}$. Observe that $j$ may belong to $\sigma([i - 1])$
and in this case $X_j^{\prime\prime}$ is independent of $X_j^{\prime}$
appearing in $X^{\sigma([i-1])}$. With this notation in mind we may
prove the following version of Lemma \ref{lem:monotonicity}.

\begin{lemma}
\label{lem:monotonicity_second}
Using the notation of Lemma \ref{lem:monotonicity}, assuming that $j$ is chosen uniformly at random from the set $[n]$ and independently of other random variables involved, we have for any $k\in [n]$,
\[
    B_k^{\prime} \le \frac{2\var(f(X))}{k}\left(\frac{n + 1}{n}\right)~,
\]
where for any $i \in [n]$,
\[
B_i^{\prime} = \EXP \left[\left(f(X)-f(X^{(j)})\right) \left(f(X^{\sigma([i-1])})-f(X^{(j) \circ \sigma([i-1])})\right)\right]~.
\]
\end{lemma}

\section{Random matrix results} 
\label{sec:rm}

In the proof of Theorem \ref{thm:main} we apply Lemma
\ref{lem:monotonicity_second} with $f$ being the top eigenvalue of a
Wigner matrix. The usefulness of this bound crucially hinges on the
fact that the variance of the
top eigenvalue is small, that is, in a sense, the top eigenvalue is
``superconcentrated''.
This fact is quantified in this section.

Our first lemma on the variance of $\lambda$ is obtained as  a combination of
a result of Ledoux and Rider \cite{MR2678393}  on Gaussian ensembles
and the universality of fluctuations for Wigner
matrices as stated in Erd\H{o}s, Yau and Yin \cite{MR2871147}.

\begin{lemma}
\label{lem:variance}
Assume that $X$ is a Wigner matrix as in Theorem \ref{thm:main}. Let $\lambda$ denote the largest eigenvalue of $X$. Then,  
\[
    \var(\lambda) \le (c  +o(1)) N^{-1/3}~,
\]
where $c>0$ is an absolute constant. 

\end{lemma}

{
\begin{remark}
The result of Lemma \ref{lem:variance} implies an improved version of the variance bound  
\[
\var(\lambda) \lesssim (\log N)^{C\log \log N}N^{-1/3},
\] 
following from \cite[Theorem 2.2]{MR2871147}. 
\end{remark}
}

We also need the following delocalization result of the top eigenvector of a Wigner matrix which can be found in Tao and Vu \cite[Proposition 1.12]{MR2669449}.

\begin{lemma}
\label{delocalization}{\em \cite{MR2669449}.}
Assume that $X$ is a Wigner matrix as in Theorem \ref{thm:main}. For
any real $c_0>0$, there exists a constant $C>0$, such that, 
with probability at least $1-C N^{-c_0}$,  any eigenvector $w$ of $X$ with $\| w \|_2  =1 $ satisfies
\[
\|w\|_{\infty} \le \frac{(\log N)^C}{\sqrt{N}}~.
\]

\end{lemma}

Our final lemma is a perturbation inequality in $\ell^\infty$-norm of
the top eigenvector of a Wigner matrix when a single entry is
re-sampled. The proof uses precise estimates on the eigenvalue
spacings in Wigner matrices proved in Tao and Vu
\cite{MR2669449} and Erd\H{o}s, Yau, and Yin \cite{MR2871147}.

\begin{lemma}
\label{fliponeelem}
 Let $X$ be a Wigner matrix as in Theorem \ref{thm:main} and $X'$ { be} an independent copy of $X$. For any $(i,j)$ with $1\le i,j \le N$. Denote by {$X^{(ij)}$} the symmetric matrix obtained from $X$ by replacing the entry $X_{ij}$ by $X'_{ij}$ and $X_{ji}$ by $X'_{ji}$.  For any $0 < \alpha < 1/10$, there exists $\kappa >0$ such that, for all $N$ large enough, with probability at least $1 - N^{-\kappa}$,
\[
\max_{1 \leq i,j \leq N} \inf_{s \in \{-1,1\}} \|s v - u^{(ij)} \|_\infty \le  N^{ - \frac 1 2 - \alpha}~,
\]
where $v$ and $u^{(ij)}$ are any unit eigenvectors corresponding to the largest eigenvalues of $X$ and {$X^{(ij)}$}.
\end{lemma}

\section{Proof of Theorem \ref{thm:main}}
\label{sec:thm1}

Now we are ready for the proof of the main results of the paper.

We start by fixing some notation. 
Let $\lambda$ denote the largest eigenvalue of the Wigner matrix $X$ of Theorem \ref{thm:main} and let $v\in S^{N-1}$ be a corresponding normalized eigenvector.
Let $k\in \left[\binom{N}{2} + N\right]$ to be specified later and let $X^{[k]}$ be the 
random symmetric matrix obtained by resampling $k$ random
entries { above} the diagonal and including the diagonal, as defined in the introduction. We denote
by $S_k \subset  \left[\binom{N}{2} + N\right]$ the set of random positions of the $k$
resampled entries. Let $\lambda^{[k]}$ denote the top eigenvalue of
$X^{[k]}$ and $v^{[k]}$ a corresponding normalized eigenvector.

For $1 \leq i \leq j \leq N$, we denote by {$Y_{(ij)}$}  the symmetric matrix obtained from $X$ by replacing
the entry $X_{ij}$ by $X^{\prime\prime}_{ij}$ where $X^{\prime\prime}$ is an independent copy of $X$. We obtain $Y^{[k]}_{(ij)}$ from $X^{[k]}$
by the same operation. We denote by $(\mu_{(ij)},u_{(ij)})$, and $(\mu^{[k]}_{(ij)},u^{[k]}_{(ij)})$ the
top eigenvalue/eigenvector pairs of $Y_{(ij)}$ and $Y^{[k]}_{(ij)}$, respectively. Let $(s, t)$ be a pair of {indices chosen uniformly at random from $\left[\binom{N}{2} + N\right]$ and satisfying $1\leq s \leq t\leq N$}.
For ease of notation, we set $Y = Y_{(st)}$, $\mu = \mu_{(st)} $ and $u = u_{(st)}$. We define similarly $Y^{[k]} = Y^{[k]}_{(st)}$, $\mu^{[k]} =  \mu^{[k]}_{(st)}$ and $u^{[k]} = u^{[k]}_{(st)}$.

By applying Lemma \ref{lem:monotonicity_second} to
the function of $n=\binom{N}{2} + N$ independent random variables 
$f\left((X_{i,j})_{1\le i\le j\le N}\right)=\lambda$, we obtain that, for
any $k\in \left[\binom{N}{2} + N\right]$, 
\begin{equation}
\label{eq:firststep}
\frac{2\var(\lambda)}{k} \cdot \frac{\binom{N}{2} + N + 1}{\binom{N}{2} + N } \ge \EXP\left[
      (\lambda - \mu)\left(\lambda^{[k]}-\mu^{[k]}\right) \right]~.
\end{equation}
{In what follows, we show that the right-hand side of \eqref{eq:firststep} satisfies
\[
\EXP\left[(\lambda - \mu)\left(\lambda^{[k]}-\mu^{[k]}\right) \right] \simeq \frac{1}{N^2}\E\left[\langle v, v^{[k]}\rangle^2\right].
\]
This relation, combined with Lemma \ref{lem:variance} and \eqref{eq:firststep}, implies
\[
\E\left[\langle v, v^{[k]}\rangle^2\right] \lesssim \frac{N^{\frac{5}{3}}}{k}~,
\]
which is sufficient for Theorem \ref{thm:main}. We proceed with the formal argument.
}
Using the notation of the previous section we have
\[
\EXP\left[(\lambda - \mu)\left(\lambda^{[k]}-\mu^{[k]}\right) \right] = \EXP\left[(\inr{v,Xv} - \inr{u,Yu})\left(\inr{v^{[k]},X^{[k]}v^{[k]}} - \inr{u^{[k]},Y^{[k]}u^{[k]}}\right)\right]~. 
\]
{ Using the fact that $v$ maximizes $\inr{v,Xv}$ and $u$ maximizes $\inr{u,Yu}$ we have 
\[
\inr{u,(X - Y)u} \le \inr{v,Xv} - \inr{u,Yu} \le \inr{v,(X - Y)v}.
\]
Observe that the elements of $X - Y$ are all zeros except at most two that correspond to resampled values. If the element $X_{t, s}$ of $X$ was resampled to get $Y$, we have, for any vector $x$,
\[
\inr{x,(X - Y)x} = U_{t, s} x_{t}x_{s}
\]
with  $U_{t,s} = (X_{t,s} - X''_{t,s}) ( 1 + \ind(t \ne s))$. Similarly, if we set $U'_{t,s} = (X'_{t,s} - X''_{t,s}) ( 1 + \ind(t \ne s))$, we have $\inr{x,(X^{[k]} - Y^{[k]})x} = U'_{t, s} x_{t}x_{s}$. Therefore, it is straightforward to see that
\begin{align*}
&(\inr{v,Xv} - \inr{u,Yu})\left(\inr{v^{[k]},X^{[k]}v^{[k]}} - \inr{u^{[k]},Y^{[k]}u^{[k]}}\right)\ge I ,
\end{align*}
where we have set,
$$
I = V_{t,s} \min\left\{v_t v_s v^{[k]}_t v^{[k]}_s, u_t u_s v^{[k]}_t v^{[k]}_s, v_t v_s u^{[k]}_t u^{[k]}_s, u_t u_s u^{[k]}_t u^{[k]}_s\right\},
$$
and for $1 \leq i \leq j \leq N$, 
$$
V_{i,j} = U_{i,j} U'_{i,j} = ( 1 + \ind(i \ne j))^2  (X_{i,j} - X''_{i,j})   (X'_{i,j} - X''_{i,j}).
$$
 In order to have some extra independence, we introduce yet another independent copy of our random variables. For $1 \leq i \leq j \leq N$, let $Z_{(ij)}$ be the symmetric matrix obtained from $X$ by replacing
the entry $X_{ij}$ by $X'''_{ij}$ where $X'''$ is an independent copy of $X$, independent of $X'$ and $X''$. We obtain $Z^{[k]}_{(ij)}$ from $X^{[k]}$
by the same operation. As above, we denote by $w_{(ij)}$, and $w^{[k]}_{(ij)}$ the
top unit eigenvector of $Z_{(ij)}$ and $Z^{[k]}_{(ij)}$, respectively. For ease of notation, with $(s,t)$ as above, we define $w = w_{(s,t)}$ and $w^{[k]} = w^{[k]}_{(st)}$. The key observation is that $V_{i,j}$ is independent of $Z_{(ij)}$ and $Z^{[k]}_{(ij)}$.}

Fix $0 < \alpha< 1/10$  and let $C$ be as in Lemma
\ref{delocalization} for $c_0 = 10$. We define {$\cE= \cE_1 \cap \cE_2$} 
to be the {intersection} of the following two events: 
\begin{itemize}
{\item $\cE_1$: for all $1 \leq i \leq j \leq N$: $\max(\|v - w_{(ij)} \|_{\infty} , \|u_{(ij)} - w_{(ij)} \|_{\infty} , \|v^{[k]} - w_{(ij)}^{[k]}\|_{\infty} , \|u_{(ij)}^{[k]} - w_{(ij)}^{[k]}\|_{\infty}  ) \le N^{-\frac 1 2 - \alpha}$.
\item $\cE_2$: $\|x\|_{\infty} \le \frac{(\log N)^C}{\sqrt{N}}$ for all $x \in \left\{v, u_{(ij)}, w_{(ij)}, v^{[k]}, u_{(ij)}^{[k]},  w_{(ij)}^{[k]} : 1 \leq i, j \leq N\right\}$.}
\end{itemize}
By Lemmas \ref{delocalization}, \ref{fliponeelem}, and the union bound, we have,  for all $N$ large enough, $\PROB(\cE_2^c) \leq N^{-6}$ and  for some $\kappa >0$, $\PROB(\cE^c) \le N^{-\kappa}$ (provided that we choose properly the $\pm$-phase for the eigenvectors $u$, $w$, $u^{[k]}$ and $w^{[k]}$). Observe that when $\cE$ holds, for all 
\begin{equation}\label{eq:xyinf}
x \in \{v_t v_s v^{[k]}_t v^{[k]}_s, u_t u_s v^{[k]}_t v^{[k]}_s, v_t v_s u^{[k]}_t u^{[k]}_s, u_t u_s u^{[k]}_t u^{[k]}_s  \}~,
\end{equation} 
we have, for all $N$ large enough, 
$$|x - w_t w_s w^{[k]}_t w^{[k]}_s| \le \frac{4(\log N)^{3C}}{N^{2 + \alpha}}~.$$ 
We show this, for brevity, only for $v_t v_s v^{[k]}_t v^{[k]}_s$. 
Denoting $\delta_t = w_t-v_t$ and $\delta_t^{[k]}=v_t^{[k]}-w_t^{[k]}$, 
we write 
\[
v_t v_s v^{[k]}_t v^{[k]}_s = (w_t -\delta_t)(w_s - \delta_s)(w^{[k]}_t - \delta_t^{[k]})(w^{[k]}_s - \delta_s^{[k]})~.
\]
Then open the brackets and use that, on $\cE$,
\[
\max\{|\delta_t|, |\delta_s|, |\delta_t^{[k]}|, |\delta_s^{[k]}|\} \le  N^{-\frac{1}{2} - \alpha}\quad \text{and}\quad \max\{|w_t|, |w_s|, |w^{[k]}_t|, |w^{[k]}_s|\} \le (\log N)^{{ 3}C} / \sqrt{N}~.
\] 
{ If $\cE$ holds, we thus have 
\begin{align*}
& I \ge  V_{t,s} w_t w_s w^{[k]}_t w^{[k]}_s   - \frac{4(\log N)^{3C}}{N^{2 + \alpha}}  |V_{t,s}| . 
%\\
%& \ge   \E_{s,t} \left(Z_{t,s}^2 v_t v_s v^{[k]}_t v^{[k]}_s \right) - c_1 \frac{(\log N)^{3C}}{N^{2 + \alpha}}
\end{align*}
%with $c_1 = \max (16, 4 \sigma_0^2)$.  
On the other hand, if $\cE_2 \backslash \cE$ holds, we get
$$
I \geq - \frac{(\log N)^{4C}}{N^2}\ind (\cE^c) |V_{t,s}|.
$$
Finally, if $\cE_2$ does not hold, using that all the
vectors are of unit norm (and therefore, $\max\{|v_t|, |v_s|, |v^{[k]}_t|, |v^{[k]}_s|\} \le 1$), we have
\begin{align*}
&  I \geq  - \ind (\cE_2^c) |V_{t,s}|~.
\end{align*}
The same bounds hold for $ V_{t,s} w_t w_s w^{[k]}_t w^{[k]}_s $ on $\cE_2 \backslash \cE$ and $\cE_2^c$. Note also that $\E V_{t,s}^2 \leq c^2_1$ for some constant $c_1 \geq 1$ depending
on $\delta$.  Combining altogether the last three bounds, by the Cauchy-Schwarz inequality, we arrive at 
$$
\E [ I ] \geq \E [ V_{t,s} w_t w_s w^{[k]}_t w^{[k]}_s ] - 4c_1\frac{(\log N)^{3C}}{N^{2 + \alpha}} - 2c_1 \frac{(\log N)^{4C}}{N^{2}} \sqrt{\PROB(\cE^c)} - 2c_1 \sqrt{\PROB(\cE_2^c)}.
$$
Recalling \eqref{eq:firststep}, we find
\begin{align*}
\E [ V_{t,s} w_t w_s w^{[k]}_t w^{[k]}_s ] &\le \frac{4\var(\lambda)}{k}+ 4c_1\frac{(\log N)^{3C}}{N^{2 + \alpha}} + 2c_1 \frac{(\log N)^{4C}}{N^{2}} \sqrt{\PROB(\cE^c)} + 2c_1 \sqrt{\PROB(\cE_2^c)}~.
\end{align*}
%Observe that $v$ depends only on $X$ as well as $v^{[k]}$ depends only on $X^{[k]}$ and both do not depend on the random choice of $(s,t)$. 
Integrating over the random choice of $(s,t)$, we have
\begin{equation}
\label{decomp}
\E [ V_{t,s} w_t w_s w^{[k]}_t w^{[k]}_s ]= \frac{1}{\binom{N}{2} + N}\E \left( \sum\limits_{1 \le i \le j \le N}V_{i,j} (w_{(ij)})_i (w_{(ij)})_j (w^{[k]}_{(ij)})_i (w^{[k]}_{(ij)})_j \right).
\end{equation}
Now, using \eqref{decomp} and using $\frac{\binom{N}{2} + N + 1}{\binom{N}{2} + N } \le 2$, we get 
\begin{equation}
\label{eq:Z1}
\E \left(\sum\limits_{1 \leq i , j \leq N} \tilde V_{i,j}(w_{(ij)})_i (w_{(ij)})_j (w^{[k]}_{(ij)})_i (w^{[k]}_{(ij)})_j \right) \leq 4 N^2 \frac{\var(\lambda)}{k}  + \eps_N ,
\end{equation}
where $\tilde V_{i,j} = V_{i,j} / 2$ if $i \ne j$, $\tilde V_{i,i} = V_{i,i}$ and 
$$
\eps_N =  4c_1\frac{(\log N)^{3C}}{N^{ \alpha}} + 2c_1 (\log N)^{4C} \sqrt{\PROB(\cE^c) } + 2c_1 N^2\sqrt{\PROB(\cE_2^c)}.
$$ Note that for $i \ne j$, $\E \tilde V_{i,j} = 2$ and $\E \tilde V_{i,i} = \sigma_0^2$. We have
\begin{equation*}
\E\left(\sum\limits_{i=1}^N  (w_{(ij)})_i (w_{(ij)})_j (w^{[k]}_{(ij)})_i (w^{[k]}_{(ij)})_j \right) \le \frac{ (\log N) ^{4C} }{N} + N \PROB(\cE_2^c).
\end{equation*}
Hence,  using that the variable $V_{i,j}$ is independent of the vectors $w_{(ij)},w_{(ij)}^{[k]}$, we deduce that 
\begin{eqnarray}
2 \E \left(\sum\limits_{1 \leq i , j \leq N}  (w_{(ij)})_i (w_{(ij)})_j (w^{[k]}_{(ij)})_i (w^{[k]}_{(ij)})_j \right) &\leq & 4 N^2  \frac{\var(\lambda)}{k}  + \eps'_N ,\label{eq:Z2}
\end{eqnarray}
where 
$$
\eps'_N = \eps_N + | 2 - \sigma_0^2|\frac{ (\log N) ^{4C} }{N} + N | 2 - \sigma_0^2| \PROB(\cE_2^c).
$$
We now argue that in \eqref{eq:Z2}, we may replace the vectors $w_{(ij)}$ and $w_{(ij)}^{[k]}$ by $v$ and $v^{[k]}$ respectively. We repeat the above argument. Recall the event $\cE = \cE_1 \cap \cE_2$ defined above. As already pointed, on the event $\cE$, we have 
$$
|v_i v_j v_i^{[k]} v_j ^{[k]}  - (w_{(ij)})_i (w_{(ij)})_j (w^{[k]}_{(ij)})_i (w^{[k]}_{(ij)})_j| \le \frac{4(\log N)^{3C}}{N^{2 + \alpha}}~. 
$$
If $\cE_2$ holds, we have 
$$
|v_i v_j v_i^{[k]} v_j ^{[k]}  - (w_{(ij)})_i (w_{(ij)})_j (w^{[k]}_{(ij)})_i (w^{[k]}_{(ij)})_j| \le \frac{2(\log N)^{4C}}{N^{2}}~.
$$
Finally, there is the deterministic bound 
$$
|v_i v_j v_i^{[k]} v_j ^{[k]}  - (w_{(ij)})_i (w_{(ij)})_j (w^{[k]}_{(ij)})_i (w^{[k]}_{(ij)})_j| \le 2.
$$
Combining the last three bounds we obtain that 
\begin{align*}
&\E \sum\limits_{1 \leq i , j \leq N}  | (w_{(ij)})_i (w_{(ij)})_j (w^{[k]}_{(ij)})_i (w^{[k]}_{(ij)})_j  - v_i v_j v^{[k]}_i v^{[k]}_j| 
\\
&\quad\leq \frac{4(\log N)^{3C}}{N^{\alpha}} + 2 (\log N)^{4C} \PROB(\cE^c) + 2N^2 \PROB(\cE_2^c).
\end{align*}
The right-hand side is upper bounded by $2 \eps_N$. We thus have proved that 
\begin{eqnarray}
2 \E \left(\sum\limits_{1 \leq i , j \leq N} v_i v_j v^{[k]}_i v^{[k]}_j \right) &\leq & 4 N^2  \frac{\var(\lambda)}{k}  + \eps''_N ,\label{eq:Z3}
\end{eqnarray}
with $\eps''_N =  \eps'_N +2 \eps_N$. As already pointed, by Lemmas \ref{delocalization}, \ref{fliponeelem}, and the union bound, we have, for all $N$ large enough,
$\PROB({\cE'_2}^c) \leq N^{-6}$ and $\PROB({\cE'}^c) \le N^{-\kappa}$. It follows that $\eps''_N\to 0$ with $N$.

Now, combining Jensen's inequality and \eqref{eq:Z3},  
\begin{align*}
\left(\E\left|\inr{v, v^{[k]}}\right|\right)^2 &\le \E\left(\sum\limits_{i = 1}^Nv_iv_i^{[k]}\right)^2
\le \E\left(\sum\limits_{1 \le i , j \le N}v_i v_j v^{[k]}_i v^{[k]}_j\right)
\le 2N^2  \frac{\var(\lambda)}{k}  + \frac{\eps''_N}{2}.
\end{align*}
From Lemma \ref{lem:variance}, the claim follows.}

\subsection{Proof of Lemma \ref{lem:monotonicity} and Lemma \ref{lem:monotonicity_second}}
\label{sec:prooflemmon}

We start with the following technical lemma. 
\begin{lemma}
\label{varianceformula}
Let $f: \X^n \to \R$ be a measurable function and let $\sigma \in S_n$
be any fixed permutation. Fix $i \in [n - 1]$ and $j \in [n]$ such that $j \notin \sigma([i])$. Let $X_1,\ldots,X_n$ be
independent random variables taking values in $\X$. Then
\begin{align*}
A_i &= \EXP \left[\left(f(X)-f(X^{(\sigma(i))})\right) \left(f(X^{\sigma([i - 1])})-f(X^{\sigma([i])})\right)\right] 
\\
&\ge \EXP \left[\left(f(X)-f(X^{(\sigma(i))})\right) \left(f(X^{\sigma([i - 1])\cup j})-f(X^{\sigma([i])\cup j})\right)\right]
\\
&\ge 0~.
\end{align*}
\end{lemma}
\begin{proof}
Without loss of generality, we may consider one particular permutation
$\sigma$, defined as follows: set $\sigma(k) = k$ for $k \notin \{1,
i\}$, 
$\sigma(i) = 1$, $\sigma(1) = i$, and we may also assume that $j = i +
1$. 
The proof is identical for any other $\sigma$ and $j$. In our case,
\[
A_i = \EXP \left[\left(f(X)-f(X^{(1)})\right) \left(f(X^{[i]\setminus\{1\}})-f(X^{[i]})\right)\right]~.
\]
Moreover, we have 
\[
\EXP \left[\left(f(X)-f(X^{(\sigma(i))})\right) \left(f(X^{\sigma([i - 1])\cup j})-f(X^{\sigma([i])\cup j})\right)\right] = A _{i + 1}.
\]
We introduce a simplifying notation. Denote $B = (X_2, \ldots,X_i)$, $B' = (X'_2, \ldots, X'_i)$ and $C = (X_{i + 2}, \ldots, X_n)$. Therefore, we may rewrite 
\[
A_i = \EXP \left[\left(f(X_1, B, X_{i + 1}, C)-f(X_1', B, X_{i + 1}, C)\right) \left(f(X_1, B', X_{i + 1}, C)-f(X_1', B', X_{i + 1}, C)\right)\right]
\] and
\[
A_{i + 1} = \EXP \left[\left(f(X_1, B, X_{i + 1}, C)-f(X_1', B, X_{i + 1}, C)\right) \left(f(X_1, B', X_{i + 1}', C)-f(X_1', B', X_{i + 1}', C)\right)\right]~.
\]
Denote $h(X_1, X_1', X_{i + 1}, C) = \EXP[\left(f(X_1, B, X_{i + 1}, C)-f(X_1', B, X_{i + 1}, C)\right)\big| X_1, X_1', X_{i + 1}, C]$. Using the independence of $B, B'$ and their independence of the remaining random variables, we have
\[
A_i = \EXP h(X_1, X_1', X_{i + 1}, C)^2~.
\]
At the same time, using the same notation for $h$ we have, by the
Cauchy-Shwarz inequality
and the fact that $X_{i + 1}$ and $X_{i + 1}'$ have the same distribution,
\begin{align*}
A_{i + 1} &= \EXP h(X_1, X_1', X_{i + 1}, C)h(X_1, X_1', X'_{i + 1},C)
\\
&=  \EXP[\EXP[ h(X_1, X_1', X_{i + 1}, C)h(X_1, X_1', X'_{i + 1},C)|X_1, X_1', C ]]
\\
&\le \EXP h(X_1, X_1', X_{i + 1}, C)^2
\\
&= A_{i}~.
\end{align*}
Now to prove that $A_{i} \ge 0$, it is sufficient to show that $A_n \ge 0$. Denoting $g(X_1) = \EXP [f(X)|\ X_1]$, we have 
\begin{align*}
A_{n} &= \EXP \left[\left(f(X)-f(X^{(1)})\right) \left(f(X^{[n]\setminus\{1\}})-f(X^{[n]})\right)\right]
\\
&= \EXP (f(X)f(X^{[n]\setminus\{1\}}) - f(X)f(X^{[n]}) - f(X^{(1)})f(X^{[n]\setminus\{1\}}) + f(X^{(1)})f(X^{[n]}))
\\
&= 2\EXP f(X)f(X^{[n]\setminus\{1\}}) - 2(\EXP f(X))^2
\\
&= 2\EXP[\EXP [f(X)f(X^{[n]\setminus\{1\}})| X_1]] - 2(\EXP f(X))^2
\\
&= 2\EXP[g(X_1)^2] - 2(\EXP f(X))^2
\\
&\ge 0~,
\end{align*}
where we used Jensen's inequality and that $\EXP g(X_1) = \EXP f(X)$.
\end{proof}

We proceed with the proof of Lemma \ref{lem:monotonicity}.

\begin{proof}
In this proof by writing $i + 1$ we mean $i + 1\ (\text{mod}\ n)$. For each permutation $\sigma \in S_n$ and fixed $i \in [n]$ we construct a corresponding permutation $\sigma'$ by defining $\sigma'(i) = \sigma(i + 1),\ \sigma'(i + 1) = \sigma(i)$ and $\sigma'(k) = \sigma(k)$ for $k \neq \{i, i + 1\}$. 

It is straightforward to see that for any fixed $i$ there is a one-to-one correspondence between $\sigma \in S_n$ and $\sigma'$. By observing that $\sigma'([i]) = \sigma([i - 1])\cup \sigma(i + 1)$ and $\sigma'([i +1]) = \sigma([i + 1])$ we have, conditionally on $\sigma$,
\begin{align*}
&\EXP \left[\left(f(X)-f(X^{(\sigma(i))})\right) \left(f(X^{\sigma([i - 1])})-f(X^{\sigma([i])})\right)\right]
\\
&=\EXP \left[\left(f(X)-f(X^{(\sigma'(i + 1))})\right) \left(f(X^{\sigma([i - 1])})-f(X^{\sigma([i])})\right)\right]
\\
&\ge \EXP \left[\left(f(X)-f(X^{(\sigma'(i + 1))})\right) \left(f(X^{\sigma'([i])})-f(X^{\sigma'([i+ 1])})\right)\right]~,
\end{align*}
where in the last step we used Lemma \ref{varianceformula}. Using the one to one correspondence between all $\sigma$ and $\sigma'$, we have
\begin{align*}
B_i &= \E_{\sigma}\EXP \left[\left(f(X)-f(X^{(\sigma(i))})\right) \left(f(X^{\sigma([i - 1])})-f(X^{\sigma([i])})\right)\right] 
\\
&= \frac{1}{n!}\sum\limits_{\sigma}\EXP \left[\left(f(X)-f(X^{(\sigma(i))})\right) \left(f(X^{\sigma([i - 1])})-f(X^{\sigma([i])})\right)\right]
\\
&\ge\frac{1}{n!}\sum\limits_{\sigma'}\EXP \left[\left(f(X)-f(X^{(\sigma'(i + 1))})\right) \left(f(X^{\sigma'([i])})-f(X^{\sigma'([i+ 1])})\right)\right]
\\
&=\E_{\sigma}\EXP \left[\left(f(X)-f(X^{(\sigma(i + 1))})\right) \left(f(X^{\sigma([i])})-f(X^{\sigma([i+ 1])})\right)\right] 
\\
&= B_{i + 1}.
\end{align*}
The proof that $B_n \ge 0$ follows from Lemma \ref{varianceformula} as well. 
\end{proof}

Finally, we prove Lemma \ref{lem:monotonicity_second}.

\begin{proof}
To prove this Lemma we show an upper bound for $B_i^{\prime}$. We have,
\begin{align*}
B_i^{\prime} &= \EXP \left[\left(f(X)-f(X^{(j)})\right) \left(f(X^{\sigma([i-1])})-f(X^{(j) \circ \sigma([i-1])})\right)\right]
\\
&=\E_{\sigma}\left(\EXP \left[\left(f(X)-f(X^{(j)})\right) \left(f(X^{\sigma([i-1])})-f(X^{(j) \circ \sigma([i-1])})\right)\big| j \in { \sigma[i - 1]}\right]\PROB(j \in { \sigma[i - 1]})\right)
\\
&\quad+\E_{\sigma}\left(\EXP \left[\left(f(X)-f(X^{(j)})\right) \left(f(X^{\sigma([i-1])})-f(X^{(j) \circ \sigma([i-1])})\right)\big| j \notin { \sigma[i - 1]}\right]\PROB(j \notin { \sigma[i - 1]})\right)~.
\end{align*}
Observe that $\PROB(j \in{ \sigma[i - 1]}) = \frac{i - 1}{n}$ and the second
summand is equal to $B_i\frac{n - i + 1}{n}$. We proceed with the
first summand. For $i \ge 1$, we have
\begin{align*}
&\E_{\sigma}\EXP \left[\left(f(X)-f(X^{(j)})\right) \left(f(X^{\sigma([i-1])})-f(X^{(j) \circ \sigma([i-1])})\right)\big| j \in { \sigma[i - 1]}\right]
\\
&= \E_{\sigma}\EXP \left[\left(f(X)-f(X^{(j)})\right) \left(f(X^{\sigma([i-1])\setminus\{j\}})-f(X^{(j) \circ \sigma([i-1])})\right)\big| j \in { \sigma[i - 1]}\right]
\\
&\quad+\E_{\sigma}\EXP \left[\left(f(X)-f(X^{(j)})\right) \left(f(X^{\sigma([i-1])})-f(X^{\sigma([i-1])\setminus\{j\}})\right)\big| j \in { \sigma[i - 1]}\right]
\\
& = B_{i - 1} - \E_{\sigma}\EXP \left[\left(f(X)-f(X^{(j)})\right) \left(f(X^{\sigma([i-1])\setminus\{j\}}) - f(X^{\sigma([i-1])})\right)\big| j \in { \sigma[i - 1]}\right]~.
\end{align*}
{ Finally, we prove that 
\begin{equation}
\label{diffeq}
\E_{\sigma}\EXP \left[\left(f(X)-f(X^{(j)})\right) \left(f(X^{\sigma([i-1])\setminus\{j\}}) - f(X^{\sigma([i-1])})\right)\big| j \in \sigma[i - 1]\right] \ge 0~.
\end{equation}
Without loss of generality, we consider a particular choice of $\sigma$ and $j$ such that $\sigma(k) = k$, for $k \in [n]$ and $j = 1$. Therefore, \eqref{diffeq} will follow from
\[
\E f(X)(f(X^{[i-1]\setminus\{1\}}) - f(X^{[i-1]})) \ge \E f(X^{(1)})(f(X^{[i-1]\setminus\{1\}}) - f(X^{[i-1]}))~.
\]
Since $X^{(1)} = (X_1^{\prime\prime}, X_2, \ldots, X_n)$, we have $\E f(X)f(X^{[i-1]}) = \E f(X^{(1)})f(X^{[i-1]})$. This implies that \eqref{diffeq} is valid whenever
\[
\E f(X)f(X^{[i-1]\setminus\{1\}}) \ge \E f(X^{(1)})f(X^{[i-1]\setminus\{1\}})~.
\]
As in the proof of Lemma \ref{varianceformula}, this relation holds due to Jensen's inequality. These lines together imply that
\[
B_i^{\prime} \le \frac{i - 1}{n}B_{i - 1} + \frac{n - i + 1}{n}B_i~,
\] 
which, using Lemma \ref{lem:monotonicity}, proves the claim.}
\end{proof}

\subsection{Proof of Lemma \ref{lem:variance}}

We start with a special case. Let us say that a Wigner matrix as in Theorem \ref{thm:main} is {\em standard} if for all $i$, $ \E X_{ii}^2= 2$. In this case, the variance of the entries of $X$ is equal to the variance of the entries of a random matrix $Y$ sampled from the Gaussian Orthogonal Ensemble (GOE). If $\mu$ is the largest eigenvalue of $Y$, it follows from \cite[Corollary 3]{MR2678393} that for some absolute constant $c >0$,
$$
\var ( \mu ) \leq c N^{-1/3}.
$$

On the other hand, it follows from \cite[Theorem 2.4]{MR2871147} (see also \cite[Theorem 1.6]{MR3034787} for a statement which can be used directly) that,
$$
N^{1/3} \left| \var ( \mu )  - \var (\lambda) \right| = o(1).
$$
We obtain the first claim of the lemma for standard Wigner matrices.  To conclude the proof of the lemma for Wigner matrices, it suffices to prove that for any Wigner matrix $X$, for some $\kappa \geq 1/3$, we have for all $N$ large enough,
\begin{equation} 
\label{eq:ll0}
\E |\lambda - \lambda_0 |^2 \leq  N^{-\kappa}.
\end{equation}
where $\lambda_0$ is the largest eigenvalue of a matrix $X_0$ obtained from $X$ by setting to $0$ all diagonal entries. We will prove it for any $\kappa < 1/2$ (an improvement of the forthcoming Lemma \ref{lem:resres0} would give \eqref{eq:ll0} for any $\kappa <1$). The proof requires some care since the operator norm of $X - X_0$ may be much larger than $1$ and the rank of $X-X_0$ could be $N$.

There is an easy inequality which is half of \eqref{eq:ll0}. Let $v_0$ be a unit eigenvector of $X_0$ with eigenvalue $\lambda_0$. We have 
$$
\lambda \geq \langle v_0 , X v_0 \rangle =  \langle v_0 , X_0 v_0 \rangle+ \langle v_0 , (X -X_0)v_0 \rangle = \lambda_0 + \sum_{i=1}^N (v_0)_i^2 X_{ii}~,
$$
where  $(v_0)_i$ is the $i$-th coordinate of $v_0$. We observe that
$v_0$ is independent of { $X_{ii}$ for all $i$ and $\E
  X_{ii} X_{jj} = 0$ for $i \neq j$. Denoting $(x)^2_+ = \max
  (x,0)^2$, by the Cauchy-Schwarz inequality,} we deduce that
$$
\E (\lambda_0 - \lambda)_+ ^2  \leq    \E \sum_{i=1}^N  (v_0)_i^4 \E X_{ii}^2 \leq \E \| v_0 \|^2_{\infty} \sigma_0^2~.
$$
We write,  $\E \| v_0 \|^2_{\infty} \leq (\log N)^ {2C}/ N +  \PROB( \|v_0 \|_{\infty} \geq (\log N)^ {C}/ \sqrt N) $. From Lemma \ref{delocalization} applied to $c_0 = 2$, we deduce that for some constant $C>0$,
$$
\E (\lambda_0 - \lambda)_+ ^2 \leq  \frac{(\log N) ^C}{N}~.
$$
It implies the easy half of \eqref{eq:ll0} for any $\kappa < 1$.

The proof of the converse inequality is more involved.  Fore ease of notation, we introduce the number for $N \geq 3$,
\begin{equation}\label{eq:defL}
L = L_N = (\log N)^{\log \log N}~.
\end{equation}
We say that a sequence of events $(A_N)$ holds {\em with overwhelming
  probability} if for any $C>0$, there exists a constant $c>0$ such
that $\PROB( A_N) \geq 1 - cN^{-C}$. We repeatedly use the fact that a polynomial { intersection} of events of overwhelming probability is an event of overwhelming probability. 
We start with a small deviation lemma which can be found, for example, in \cite[Appendix B]{MR2981427}.

\begin{lemma}\label{lem:dev}
Assume that $(Z_i)$ $1 \leq i \leq N$ are independent centered complex variables such that for some $\delta >0$, for all $i$, $\E \exp \left( |Z_i|^{\delta}\right) \leq 1/ \delta$. Then, for any $(x_i) \in \mathbb{C}^N$ 
with overwhelming probability,
$$
\left| \sum_{i=1}^N x_i Z_i \right| \leq L \| x\|_2~.
$$
\end{lemma}

For $z = E + {\mathbf{i}}\eta$ with $\eta >0$  and $E \in \mathbb{R}$, we introduce the resolvent matrices 
$$
R (z ) = ( X - z I) ^{-1} \hbox{ and } R_0 (z ) = ( X_0 - z I) ^{-1}~,
$$
where $I$ denotes the identity matrix.
The following lemma asserts that the resolvent can be used to estimate the largest eigenvalue of $X$ and $X_0$.
\begin{lemma}\label{lem:resl}
Let $X$ be a  Wigner matrix as in Theorem \ref{thm:main} and let
$\lambda_1 \geq \ldots \geq \lambda_N$ be its  eigenvalues. For any $1 \leq k \leq N$, there exists an integer $1 \leq i \leq N$ such that for all $E$ and $\eta >0$
$$
\frac 1 2 \max( \eta, |\lambda_k - E |) ^{-2} \leq N\eta ^{-1} \Im R( E + {\mathbf{i}}\eta)_{ii}~. 
$$
Moreover, let $1 \leq k \leq L$.  There exists  $c_0 >0$ such that with overwhelming probability, we have $|\lambda_k  - 2 \sqrt N| < L^{c_0} N^{-1/6}$ and for all integers $ 1 \leq i \leq N$, and all $E$ such that  $|E - 2 \sqrt N| <  L^{c_0} N^{-1/6}$, 
$$
 N\eta ^{-1} \Im R( E + {\mathbf{i}}\eta)_{ii} \leq L^{c_0} \min_{ 1 \leq j \leq N} (\lambda_j -E )^{-2}~. 
$$
\end{lemma} 
\begin{proof}
From the spectral theorem, we have
$$
\Im R_{ii} = \sum_{p = 1}^N \frac{ \eta (v_p)_i ^2}{ (\lambda_p - E )^2 + \eta ^2}~,   
$$
where $(v_1, \ldots, v_N)$ is an orthonormal basis of eigenvectors of $X$ and $(v_p)_i$ is the $i$-th coordinate of $v_p$. In particular, 
$$
N \eta^{-1} \Im R_{ii} \geq  \frac{ N  ({v_p})_i ^2}{ (\lambda_k- E )^2 + \eta ^2} \geq  \frac{ N ({v_p})_i ^2}{ 2\max( \eta, |\lambda_k - E |) ^{2}  }~. 
$$
From the pigeonhole principle, for some $i$, $({v_p})_i^2 \geq 1/N$ and the first statement of the lemma follows.

Fix an integer $1 \leq k \leq L$.  From \cite[Theorem 2.2]{MR2871147} and Lemma \ref{delocalization}, for some constants $c_0,C_0>0$, we have, with overwhelming probability, that  the following event $\cE$ holds: $|\lambda_k- 2 \sqrt N| \leq L^{c_0} N ^ {-1/6}$, for all integers $1 \leq p \leq N$, 
$$
\lambda_p \leq 2 \sqrt N - 2 C_0 p^{2/3} N^{-1/6} + L^c p^{-1/3}N^{-1/6}~,
$$
 and $ \| v_p \|^2_{\infty} \leq L / N$. We set $q = \lfloor C L^{3c_0} \rfloor$ for some $C$. Let $E $ be such that $|E - 2 \sqrt N| \leq L^{c_0} N ^ {-1/6}$. On the event $\cE$, if $C$ is large enough, we have, for all $p > q$, $E - \lambda_p  \geq C_0 p^{2/3} N^{-1/6}$ and
$$
\sum_{p = q+1}^{N} \frac{ N  (v_p)_i^2}{ (\lambda_p - E )^2 + \eta ^2} \leq \sum_{p = q+1}^{N} \frac{L}{ (\lambda_p - E )^2}  \leq  \frac 1 {C^2_0}   \sum_{p = q+1}^{N} \frac{L N^{1/3}} {p ^{4/3}} \leq c_1 L  N^{1/3} q^ {-1/3}.
$$
On the other hand, on the same event $\cE$, we have
$$
\sum_{p = 1}^{q} \frac{ N  (v_p)_i ^2}{ (\lambda_p - E )^2 + \eta ^2} \leq \sum_{p = 1}^{q} \frac{ N  (v_p)_i ^2}{ \min_{1 \leq j\leq N} (\lambda_j - E )^2} \leq \frac{L q }{\min_{1 \leq j\leq N} (\lambda_j  - E)^{2}}~.
$$
It remains to adjust the value of the constant $c_0$ to conclude the proof.
\end{proof}

The next step in the proof of \eqref{eq:ll0} is a comparison between the resolvent of $X$ and $X_0$ for $z$ close to $2\sqrt N$. The following result is a corollary of \cite[Theorem 2.1 (ii)]{MR2871147}.

\begin{lemma}\label{loclaw}
Let $X$ be a Wigner matrix as in Theorem \ref{thm:main}. There exists $c >0$ such that, with overwhelming probability, the following event holds:  for all $z = E + {\mathbf{i}}\eta$ such that $|2 \sqrt N - E| \leq  \sqrt N$ and $N^{-1/2} L ^c \leq  \eta \leq N^{1/2}$, all $i \ne j$, we have  
$$
| R(z)_{ij}   | \leq \Delta 
\quad  \hbox{ and } \quad | R(z)_{ii}   | \leq c N^{-1/2},
$$
where $\Delta  = L^c  (|E-2\sqrt N| + \eta)^{1/4}  N^{-7/8} \eta^{-1/2} + L^c N^{-2} \eta^{-1} $. 
\end{lemma}
\begin{proof}
Let $Y = X/ \sqrt N $ and for $z \in \mathbb{C}$, $\Im (z) >0$,  $G (z) = ( Y - z I)^{-1}$. We have $R(z) = N^{-1/2} G(z N^{-1/2})$. Theorem 2.1 (ii) in  \cite{MR2871147} asserts that  with overwhelming probability for all $w= a + {\mathbf{i}} b$ such that $|a| \leq 5$ and $N^{-1} L ^c \leq b \leq 1$, all $i \ne j$, we have  
$$
| G(w)_{ij}  | \leq \delta 
\quad  \hbox{ and } \quad | G(w)_{ii}  - m (w) | \leq  \delta~,
$$
where $\delta  = L^c  \sqrt{\Im (m (w))  / (N b)} + L^c (Nb)^{-1} $ and $m(w)$ is the Cauchy-Stieltjes transform of the semi-circular law (for its precise definition see \cite{MR2871147}).  Then \cite[Lemma 3.4]{MR2871147} implies that, for some $C >0$,  for all $w  = a + {\mathbf{i}}b$, $|a| \leq 5$ and $0 \leq b \leq 1$, we have $|m(w)| \leq C$ and $|\Im( m(w)) | \leq  C \sqrt{ |a - 2| + b}$. We apply the above result for $a = E / \sqrt N$ and $b = \eta / \sqrt N$. We obtain the claimed statement for $R(z) = N^{-1/2} G(z N^{-1/2})$. 
\end{proof}

We use Lemma \ref{loclaw} to estimate the difference between $R(z)$ and $R_0(z)$. 

\begin{lemma}\label{lem:resres0}
Let $X$ be a Wigner matrix as in Theorem \ref{thm:main}, let $X_0$ be
obtained from $X$ by setting to $0$ all diagonal entries, and let
$c_0$ be as in  Lemma \ref{lem:resl}. With overwhelming probability, the following event holds:  for all $z = E + {\mathbf{i}}\eta$ such that $|2 \sqrt N - E| \leq  L^{c_0} N^{-1/6}$ and  $  \eta  =  N^{-1/4}$, all $i$,
 $$
| R_0 (z)_{ii} - R(z)_{ii} | \leq \frac{1}{4 N \eta}~.
$$
\end{lemma}
\begin{proof}
The resolvent identity states that if $A-z I$ and $B-z I $ are invertible matrices then 
\begin{equation}\label{eq:resid}
(A - zI)^{-1} = (B-zI)^{-1} + (B - zI)^{-1}(B-A)(A - zI)^{-1}~.
\end{equation}
Applying twice this identity, it implies that 
$$
R = R_0 + R_0 (X_0-X)R_0 + R_0 (X_0-X)R_0 (X_0-X)R
$$
(where we omit to write the parameter $z$ for ease of notation). For any integer $1 \leq i \leq N$, we thus have 
$$
R_{ii} - (R_0)_{ii} =  -\sum_{j} (R_0)_{ij}  X_{jj} (R_0)_{ji} + \sum_{j,k} (R_0)_{ij} X_{jj} (R_0)_{jk} X_{kk}   R_{ki} = - I(z) + J(z)~.
$$
Note that $X_{jj}$ is independent of $R_0$. By Lemma \ref{lem:dev} and
Lemma \ref{loclaw} we find that, with overwhelming probability,
$$
|I(z)| \leq  L \Delta^2 \sqrt N + c L N^{-1}~. 
$$
For a given $z = E + {\mathbf{i}} \eta$ such that $|E - 2 \sqrt N| \leq L^{c_0} N^{-1/6}$ and $\eta = N^{-1/4}$, it is straightforward to check that, for some $c>0$, {$\Delta \le L^c N^{-19/24}$} and $|I(z)|\leq L^c N^{-13/12} =  o(1/(N\eta))$. 

Similarly, we have
$$
|J(z)| \leq \sum_{k} |X_{kk}| |R_{ki}|  |G_k| \quad \hbox{ with } \; G_k =  \sum_{j} (R_0)_{ij} X_{jj} (R_0)_{jk}~. 
$$
For a given $z$, by Lemma \ref{lem:dev} and Lemma \ref{loclaw}, we have with overwhelming probability, for all $k$, $|G_k| \leq L^c N^{-13/12}$ and $| J(z) | \leq   L (\Delta N  + c N^{-1/2} ) L^c N^{-13/12} =o(1/(N\eta))$.

For a given $z$, let $\mathcal E_z$ be the event that  $\max_{1 \leq i \leq N} |R(z)_{ii} - R_0(z)_{ii}| \leq (8 N \eta)^ {-1}$ and $\mathcal E'_z$ the event that $\max_{1 \leq i \leq N}|R(z)_{ii} - R_0(z)_{ii}| \leq (4 N \eta)^ {-1}$. We  have proved so far that for a given $z = E + {\mathbf{i}}\eta$ such that $|E - 2 \sqrt N| \leq L^{c_0} N^{-1/6}$ and $\eta = N^{-1/4}$,  with overwhelming probability, $\mathcal E_z$ holds.  By a net argument, it implies that with overwhelming probability, the events $\mathcal E_z'$  hold  jointly for all $z = E + {\mathbf{i}} \eta$ with $|E - 2 \sqrt N| \leq L^{c_0} N^{-1/6}$ and $\eta = N^{-1/4}$. Indeed, from the resolvent identity \eqref{eq:resid}, we have $|R_{ij} (E + {\mathbf{i}} \eta) - R_{ij} ( E' + {\mathbf{i}} \eta) | \leq \eta^{-2} |E - E'|$. It follows that if  $|E - E'| \leq \eta^2 (8 N \eta)^ {-1} \leq  N^{-1} $ then  $|R(z)_{ii} - R_0(z)_{ii}| \leq (8 N \eta)^ {-1}$.  Let $\mathcal N$ be a finite subset of the interval $K = \{ E : |E - 2 \sqrt N| \leq L^{c_0} N^{-1/6}\}$ such that for all $E \in K$, $\min_{E'  \in \mathcal N} |E - E'| \leq N^{-1}$.  We may assume that $\mathcal N$ has at most $N$ elements. From what precedes we have the inclusion, with $\eta = N^{-1/4}$,
$$
 \bigcap_{z = E + {\mathbf{i}} \eta : E \in \mathcal N} \mathcal E_z \; \subseteq \bigcap_{z = E + {\mathbf{i}}\eta : E \in K} \mathcal E'_z~. 
$$
From the union bound, the right-hand side holds with overwhelming probability. It concludes the proof of the lemma. \end{proof}

Now we have all ingredients necessary to conclude the proof of
\eqref{eq:ll0}. 
Let $\eta = N^{-1/4}$. We prove that for some $c >0$, with overwhelming probability,  
$$
\lambda \leq \lambda_0 + L^c \eta~.
$$
By Lemma \ref{lem:resl}, with overwhelming probability,  $|\lambda - 2 \sqrt N | \leq L^{c_0} N^{-1/6}$ and for some $j$, 
$$
N \eta^{-1} \Im R (\lambda + {\mathbf{i}} \eta)_{jj} \geq \frac 1 2 \eta^{-2}. 
$$
and if $\lambda > \lambda_0$, 
$$
N \eta^{-1} \Im R_0 (\lambda + {\mathbf{i}} \eta)_{jj} \leq L^{c_0} (\lambda - \lambda_0)^{-2}~. 
$$
By Lemma \ref{lem:resres0}, we deduce that  with overwhelming probability, if $\lambda > \lambda_0$, 
$$
\frac 1 4 \eta^{-2} \leq L^{c_0} (\lambda - \lambda_0)^{-2}~. 
$$
Hence, $\lambda \leq \lambda_0 + 2 L^{c_0/2} \eta$, concluding the proof of  \eqref{eq:ll0}.

\subsection{Proof of Lemma \ref{fliponeelem}}

Let $\lambda = \lambda_1 \geq \cdots \geq \lambda_N$ be the eigenvalues of $X$.  For any $(i,j)$, let $\lambda^{(ij)}$ be the largest eigenvalue of ${X^{(ij)}}$. We start by proving that $\lambda$ and $\lambda^{(ij)}$ are close compared to their fluctuations. We have 
$$
\lambda \geq \langle u^{(ij)} , Xu^{(ij)} \rangle = \lambda^{(ij)} + \langle u^{(ij)} , (X - {X^{(ij)}} ) u^{(ij)} \rangle  \geq  \lambda^{(ij)} - 2 (|X_{ij}| + |X'_{ij}|) \| u^{(ij)} \|^2_{\infty}~,
$$
{ where $u^{(ij)}$ is as in Lemma \ref{fliponeelem}.}
 Since $X$ and ${X^{(ij)}}$ have the same distribution, we deduce from
 Lemma \ref{delocalization} that, { for any $c_0 >0$, there exists $C > 0$ such that with probability at least $1 - CN^{2-c_0}$,  $\| v \|_{\infty} \leq (\log N)^C / \sqrt N$ and $\max_{ij} \| u^{(ij)} \|_{\infty} \leq (\log N)^C / \sqrt N$. For all $N$ large enough, we have $(\log N)^C  \leq L$, where $L$ is defined in (\ref{eq:defL}). Hence for any $c_0$, for some new constant $C>0$, with probability at least $1 - CN^{2-c_0}$,  $\| v \|_{\infty} \leq L/ \sqrt N$ and $\max_{ij} \| u^{(ij)} \|_{\infty} \leq L/ \sqrt N$. Since $c_0$ can be taken arbitrarily large, we deduce that}  with overwhelming probability, $\| v \|_{\infty} \leq L / \sqrt N$,
$\max_{ij} \| u^{(ij)} \|_{\infty} \leq L/ \sqrt N$ and $\max_{ij}(
|X_{ij}| + |X'_{ij}|) \leq L/2$. On this event, we get 
$$
\lambda \geq  \lambda^{(ij)}  - \frac{L^{3}} {N}~.
$$
Reversing the role $X$ and ${X^{(ij)}}$  and using the union bound, we deduce that, with overwhelming probability,
$$
\max_{ij} |\lambda  -  \lambda^{(ij)} | \leq \frac{L^{3}} {N}~.
$$
It follows from \cite[Theorem 1.14]{MR2669449} that, for any 
$\rho >0$, 
there exists $\kappa >0$ such that,  for all $N$ large enough, 
$$
\PROB ( \lambda_2  <  \lambda - N^{-1/2-\rho} ) \geq 1 - N^{-\kappa}~.
$$
Let $(v_1, \ldots, v_p)$ be an orthonormal basis of eigenvectors of
$X$ associated to the eigenvalues $(\lambda_1, \ldots, \lambda_N)$
with $v_1  = v$. We set $\theta = 2/5 - 3 \rho / 5$ and $q = \lfloor
N^{\theta} \rfloor$. For some constant $c>0$ to be defined and $\rho \in (0,1/16)$, we introduce the event $\cE_\rho$ such that
\begin{itemize}
\item 
$\lambda_2  <  \lambda -  N^{-1/2 - \rho}$ and $ \lambda_q \leq  \lambda - c q^{2/3} N^{-1/6}$~;
\item
$\max_{1 \leq p \leq q} \| v_p \|_{\infty} \leq L / \sqrt N$ and $\max_{ij} \| u^{(ij)} \|_{\infty} \leq L / \sqrt N$~; 
\item
$\max_{ij}( |X_{ij}| + |X'_{ij}|) \leq L/2$~.
\end{itemize}
From what precedes, Lemma \ref{delocalization} and \cite[Theorem 2.2]{MR2871147}, for some $c$ small enough, for any $\rho >0$ there exits $\kappa >0$ such that for all $N$ large enough, $\PROB(\cE_\rho) \geq 1 - N^{-\kappa}$. Note also, that we have checked that if $\cE_\rho$ holds then $\max_{ij} |\lambda  -  \lambda^{(ij)} | \leq  L^{3} / N$.

On the event $\cE_\rho$, we now prove that $v$ and $u^{(ij)}$ are
close in $\ell^{\infty}$-norm. For a fixed $(i,j)$, we write,
$u^{(ij)} = \alpha v + \beta x + \gamma y$, where $\alpha^2 + \beta ^2
+ \gamma^2 = 1$ with $\beta,\gamma$ non-negative real numbers, $x$ is a
unit vector in the vector space spanned by $(v_2, \ldots, v_q)$, and $y$ is a unit vector in the vector space spanned by $(v_{q+1}, \ldots, v_N)$. Set 
\[
w = {(X^{(ij)} - X)} u^{(ij)} + ( \lambda - \lambda^{(ij)} ) u^{(ij)}.
\]
We have
$$
\lambda u^{(ij)} = \alpha \lambda v + \beta X x + \gamma X y +  w~. 
$$
Taking the scalar product with $y$, we find
$$
\lambda \gamma = \lambda \langle y , u^{(ij)} \rangle = \gamma \langle y , X y \rangle  + \langle y, w\rangle \leq (\lambda - c q^{2/3} N^{-1/6})\gamma  +  \langle y, w\rangle~.
$$
Hence, 
%G
$$
\gamma\leq  c^{-1} q^{-2/3} N^{1/6} \| w\|_2 \leq
 c^{-1} q^{-2/3} N^{1/6} \left( \frac{L^{2} }{\sqrt N}   + 
 \frac{L^{4}}{N^{3/2}}   
\right) \leq 2  c^{-1} L^{4}  N^{-2\theta / 3 - 1/3}~.
$$
Similarly, taking the scalar product with $x$, we find
$$
\beta \leq N^{1/2 + \rho} \langle x, w \rangle \leq  N^{1/2 + \rho} \left( \left| \langle x,  (X - {X^{(ij)}}) u^{(ij)} \rangle\right|  +  \frac{L^{ 3}}{N} \right)~.
$$
Since $|\langle a , b \rangle| \leq \| a\|_{\infty} \|b \|_1 \leq m  \|a \|_{\infty} \|b\|_{\infty} $ where $m$ is the number of non-zeros entries of $b$, we have $\left| \langle x,  (X - {X^{(ij)}}) u^{(ij)} \rangle\right| \leq  \| x \|_{\infty} L^{2} / \sqrt N$. By construction, $x  = \sum_{p=2}^q \gamma_p v_p$ where $\sum_p |\gamma_p|^2 = 1$. If $\cE_\rho$ holds, { using the Cauchy-Schwarz inequality and $ \| v_p \|_{2} = 1$,} we deduce that 
$$
\| x \|_{\infty} \leq \sum_{p=2}^q |\gamma_p| \| v_p \|_{\infty} \leq \frac{L}{\sqrt N} \sum_{p=2}^q |\gamma_p| \leq \frac{L \sqrt q}{\sqrt N} \leq L N^{\theta / 2 - 1/2}~. 
$$
So finally, 
$$
\beta \leq 2   L^{3} N^{-1/2 + \theta /2 + \rho}~,
$$
We deduce that $|\alpha| = \sqrt{1 - \beta^2 - \gamma^2} \geq 1 - \beta - \gamma$ is positive for all $N$ large enough. We set $s = \alpha / |\alpha|$. We find, since $\|y\|_{\infty} \leq \| y\|_2 \leq 1$,
$$
\| s v - u^{(ij)}\|_{\infty} \leq ( 1 - |\alpha| ) \| v \|_{\infty} + \beta \|x \|_{\infty} +  \gamma \leq  2\beta \|x \|_{\infty} +2\gamma~.
$$
For our choice of $\theta= 2/5 - 3 \rho / 5$,  this last expression is
$O( L^{4} N^{-3/5 + 8 \rho /5})$. {Indeed, we have 
\begin{align*}
&\gamma \leq  2c^{-1} L^4 N^{-4/15+2\rho /5-1/3}=  2c^{-1} L^4 N^{-3/5+2\rho /5}, \\
&\|x \|_{\infty} \leq  L N^{1/5-3\rho /10-1/2} = LN^{-3/10-3\rho /10},\\
&\beta  \leq  2 L^3 N^{-1/2+1/5-3\rho/10+\rho} = 2L^3 N^{-3/10+7\rho/10}.
\end{align*}}
Since $\rho<1/16$, we have $3/5 - 8 \rho /5 > 1/2$. Hence, finally, if we set ${\kappa^{\prime}} = 1/10 - 8  \rho/ 5 >0$, we get that  $\| s v - u^{(ij)}\|_{\infty}  = O ( L^4 N^{-1/2{ - \kappa^{\prime}}})$. This concludes the proof of the lemma.

\section{Proof of Theorem \ref{thm:main2}}
\label{sec:thm2}
{  
The proof of Theorem \ref{thm:main2} relies on the rigorous justification of the heuristic argument sketched below the statement of Theorem \ref{thm:main2}, see the forthcoming Lemma \ref{lem:llk}. This is performed by a careful perturbation argument on the resolvent in Lemma \ref{lem:resresk}. Indeed, the resolvent has nice analytical properties and it is intimately connected to the spectrum, as illustrated in Lemma \ref{lem:reslk}}. 

Recall that $S_k=\{(i_1,j_1),\ldots,(i_k,j_k)\}$
is the set of $k$ pairs 
chosen uniformly at random (without replacement) from the set of all ordered pairs $(i,j)$ of indices with $1\le
i\le j\le N$ which is used in the definition of $X^{[k]}$. We denote by $\lambda$ and $\lambda^{[k]}$ the largest eigenvalues of $X$ and $X^{[k]}$. Recall the definition of $L = L_N$ in \eqref{eq:defL} and the notion of overwhelming probability immediately below \eqref{eq:defL}. The main technical lemma is the following:
\begin{lemma}\label{lem:llk}
Let $X$ be a Wigner matrix as in Theorem \ref{thm:main2} and let $\lambda = \lambda_1 \geq \cdots \geq \lambda_N$ be its eigenvalues. For any $c>0$  there exists a constant $c_2 >0$ such that for all $\eps >0$, for all $N$ large enough, with probability at least $1- \eps$,
$$\max_{k \leq N^{5/3} L^{-c_2}}\max_{p \in \{1,2\}} | \lambda_p - \lambda_p^{[k]}| \leq N^{-1/6} L^{-c}~.$$
\end{lemma}

We postpone the proof of Lemma \ref{lem:llk} to the next subsection.  We denote by $R (z)  = (X - z I)^{-1}$ and $R^{[k]} (z) = (X^{[k]}- z I)^{-1}$ the resolvent of $X$ and $X^{[k]}$.   The proof of Lemma \ref{lem:llk} is based on this comparison lemma on the resolvents. 

\begin{lemma}\label{lem:resresk}
Let $X$ be a Wigner matrix as in Theorem \ref{thm:main}. Let $c_0 >0$ be as in Lemma \ref{lem:resl} and let $c_1 >0$. There exists $c_2>0$ such that, with overwhelming probability, the following event holds:  for all $k \leq N^{5/3} L^{-c_2}$, for all $z = E + {\mathbf{i}} \eta$ such that $|2 \sqrt N - E| \leq  L^{c_0} N^{-1/6}$ and  $  \eta  =  N^{-1/6} L^{-c_1}$,
 $$
\max_{1 \leq i,j \leq N} N \eta | R^{[k]} (z)_{ij} - R(z)_{ij} | \leq \frac{1}{L^2}~.
$$
\end{lemma}
We postpone the proof of Lemma {\ref{lem:resresk}} to the next subsection. Our next lemma connects the resolvent with eigenvectors. 
\begin{lemma}\label{lem:reslk}
Let $X$ be a Wigner matrix as in Theorem \ref{thm:main} and let $\eps
>0$. There exist $c_1,c_2$ such that the following event holds for all $N$ large enough with probability at least $1 - \eps$:  for all $k \leq N^{5/3} L^{-c_2}$, we have, with $z = \lambda + {\mathbf{i}}\eta$, $ \eta  =  N^{-1/6} L^{-c_1}$,
 $$
\max_{1 \leq i,j \leq N} |  N \eta  \Im R(z)_{ij} - N v_i v_{j} | \leq \frac 1 {L^2} \quad \hbox{ and } \quad 
\max_{1 \leq i,j \leq N} | N\eta \Im R^{[k]} (z )_{ij} - N v^{[k]}_i v^{[k]}_{j} | \leq \frac 1 {L^2}~.
$$
\end{lemma}

\begin{proof}
Let $\lambda  = \lambda_1  \geq \lambda_2 \geq \cdots \geq \lambda_N$ be the eigenvalues of $X$. Let $(v_1, \ldots,  v_N)$ be an eigenvector basis of $X$. Recall that 
$$
N \eta  \Im R(E + {\mathbf{i}} \eta )_{ij}  = \sum_{p=1}^N \eta^2 \frac{ N  (v_p)_i (v_p)_j }{ (\lambda_p - E )^2 + \eta ^2}~. 
$$
As in the proof of Lemma \ref{lem:resl}, from
\cite[Theorem 2.2]{MR2871147} and Lemma \ref{delocalization}, for some constants $c_0,C>0$, we have with overwhelming probability that  the following event $\cE$ holds: $|\lambda- 2 \sqrt N| \leq L^{c_0} N ^ {-1/6}$, for all integers $1 \leq p \leq N$, $ \| v_p \|^2_{\infty} \leq L / N$ and for all  $q >p$ with $q = \lfloor L^{c_0} \rfloor$ and $E$ such that $|E - 2 \sqrt N| \leq L^{c_0} N^{-1/6}$ we have 
$$
\sum_{p = q+1}^{N} \frac{ N  (v_p)_i (v_p)_j }{ (\lambda_p - E )^2 + \eta ^2}  \leq C L  N^{1/3} q^ {-1/3}~.
$$
On the other hand,  let $\cE_\delta$ be the event that $\lambda_2 \leq \lambda - \delta N^{-1/6}$. Fix $\eps >0$. From  \cite[Theorem 2.7]{MR3253704} and, e.g., \cite[Chapter 3]{MR2760897}, there exists $\delta >0$ such that 
$$
\PROB (\cE_\delta ) \geq 1 - \eps~. 
$$
On the event $\cE\cap \cE_\delta$, if $| \lambda - E| \leq (\delta/2) N^{-1/6}$, we have
$$
\sum_{p = 2}^{q} \frac{ N (v_p)_i (v_p)_j }{ (\lambda_p - E )^2 + \eta ^2} \leq \frac 4 {\delta^2}L q N^{1/3}~.
$$
Finally, if $|\lambda - E | \leq \eta/L^2$, on the event $\cE$, we find easily, if $v_i$ is i-th coordinate of $v$,
$$
 \left|  \eta^2 \frac{ N  v_i v_j }{ (\lambda - E )^2 + \eta ^2} -  N v_i v_j  \right| \leq \frac 1 {L^3}~. 
$$
For some $c_1 >0$, we thus find, that if $\eta = N^ {-1/6} L^{-c_1}$ then on the event $\cE\cap \cE_\delta$, for all $E$ such that $|\lambda - E | \leq \eta/L^2$ we have 
$$
\max_{i ,j}|  N \eta  \Im R(E + {\mathbf{i}} \eta )_{ij} - N v_i v_{j} | \leq \frac 1 {L^2}~. 
$$
We apply this last estimate $R$ and $E = \lambda$. For each $k$, let $\cE^ {[k]}$ be the  event corresponding to $\cE$ for $X^{[k]}$ instead of $X$. We apply the above estimate on the event $\cE'_k =  \cE^{[k]} \cap \cE_\delta \cap \{ \max_{p = 1 ,2} |\lambda_p - \lambda_p^{[k]}|  \leq \eta / L^2\}$ to  $R^{[k]}$ and $E = \lambda$. By Lemma \ref{lem:llk} and the union bound $\cap_{k \leq N^{5/3} L^{-c_2}} \cE'_k$ has probability at least $1 - 2 \eps$. It concludes the proof. 
\end{proof}

We may now conclude the proof of Theorem \ref{thm:main2}. Let $c_1,c_2$ be as in Lemma \ref{lem:reslk}, $k \leq N^{5/3} L^{- c_2}$ and $\eta = N^{-1/6} L^{-c_1}$. Up to increasing the value of $c_2$, we may also assume that the conclusion of Lemma \ref{lem:resresk} holds.  By  Lemma \ref{delocalization}, Lemma \ref{lem:resresk} and Lemma \ref{lem:reslk}, for any $\eps >0$, for all $N$ large enough, with probability at least $1 - \eps$, it holds that for some $c >0$: $ \sqrt N \| v\|_{\infty} \leq (\log N) ^c$, $\sqrt N \| v ^{[k]} \|_\infty \leq (\log N)^c$ and 
$$
\max_{i ,j}|   N v_i v_{j} -  N v^{[k]}_i v^{[k]}_{j} | \leq \frac 3 {L^2}.
$$
Applied to $i = j$, we get that for some $s_i \in \{-1,1\}$, 
$$
\sqrt N | s_i v_i - v_i^{[k]} | \leq \frac{\sqrt 3}{ L}. 
$$
Notably, we find 
$$
(1 - s_i s_j)  N | v_i v_j | \leq |   N v_i v_{j} -  N v^{[k]}_i v^{[k]}_{j} |  +  \frac {2 \sqrt 3} {L} (\log N) ^{c}   \leq   \frac 4 {L} (\log N) ^{c}.
$$
Let $J = \{ 1 \leq i \leq N : \sqrt N |v_i| \geq  L^{-1/3} \}$. It follows from the above inequality that for $i,j \in J$, $s_i = s_j$. Let $s$ be this common value.  We have for all $i \in J$,  
$$
\sqrt N | s v_i - v_i^{[k]} | \leq \frac{\sqrt 3}{ L}. 
$$
Moreover, for all $i \notin J$, by definition, 
$$
\sqrt N | s v_i - v_i^{[k]} | \leq  \sqrt N |v_i| + \sqrt N |v^{[k]}_i| \leq  L^{-1/3} + L^{-1/3} + \sqrt 3 L ^{-1}. 
$$
It concludes the proof of  Theorem \ref{thm:main2}.

\subsection{Proof of Lemma \ref{lem:resresk}}

{ 
The proof of Lemma \ref{lem:resresk} is based on a technical martingale argument. Thanks to the resolvent identity \eqref{eq:resid}, we will write $R^{[k]}_{ij}(z) - R_{ij}(z)$ as a sum of martingale differences up to small error terms, this is performed in Equation \eqref{eq:RkRij}. These martingales will allow us to use concentration inequalities. Each term of the martingale differences will be estimated thanks to the upper bound on resolvent entries given in Lemma \ref{loclaw}.}

{ We apply many times the resolvent identity and for technical convenience, it will be easier to have a uniform bound on our random variables. }We thus start by truncating our random variables $(X_{ij})$. Set $\tilde X_{ij} = X_{ij} \IND ( |X_{ij} | \leq (\log N)^{c})$ and $\tilde X'_{ij} = X'_{ij} \IND ( |X'_{ij} | \leq (\log N)^{c})$ with $c= 2/ \delta$. The matrix $\tilde X$ has independent entries above the diagonal. Moreover, since $\EXP \exp ( |X_{ij}|^\delta) \leq 1/ \delta$, with overwhelming probability, $X = \tilde X$ and $X' = \tilde X'$. It is also straightforward to check that $\EXP  |X_{ij}|^2 \IND ( |X_{ij} | \geq (\log N) ^{c} ) = O(   \exp ( -  (\log N)^{2}/2))$. It implies that $|\EXP \tilde X_{ij}| = O (  \exp ( - ( \log N)^2 /2 )$ and $\var (\tilde X_{ij}) = 1 + O ( \exp ( - (\log N)^2 /2))$ for $i \ne j$. We define the matrix $\bar X$ with for $i \ne j$, $$\bar X_{ij} =( \tilde X_{ij} -  \EXP \tilde X_{ij} )  / \sqrt{\var (\tilde X_{ij} )}\quad \hbox{ and } \quad \bar X_{ii}  = \tilde X_{ij} -  \EXP \tilde X_{ij}.$$ 
The matrix $\bar X$ is a Wigner matrix as in Theorem \ref{thm:main2} with  entries in $[-L/4,L/4]$. Moreover,  from Gershgorin's circle theorem \cite[Theorem 6.6.1]{HJ85},  with overwhelming probability, the operator norm of $X - \bar X$ satisfies $\| X - \bar X \| = O ( N \exp ( - ( \log N)^2/2 )$. Observe that from the spectral theorem, for any Hermitian matrix $A$, $\| (A-z)^{-1}\| \leq |\Im(z)|^ {-1}$. In particular, from the resolvent identity \eqref{eq:resid}, we get $\| (X-z)^{-1} - (\bar X - z)^{-1}\| =  \| (X-z)^{-1} (X - X') (\bar X - z)^{-1}\|\leq \Im (z)^{-2} \| X - \bar X \|   = O ( N^3 \exp ( - ( \log N)^2/2 )$ if $\Im (z) \geq N^{-1}$.  The same truncation procedure applies for $X^{[k]}$. In the proof of Lemma \ref{lem:resresk},  we may thus assume without loss of generality that the random variables $X_{ij}$ have support in $[-L/4,L/4]$. 

{ It will also be convenient to assume that the random subset $S_k$ does not contain too many points on a given row or column. To that end,} for $0 \leq  t \leq k$, let $\cF_t$ be the $\sigma$-algebra generated by the random variable $X$, $S_{k}$ and $(X'_{i_s,j_s})_{1 \leq s \leq t}$. For $1 \leq i,j \leq N$, we set $$T_{ij} = \{ t : \{ i_t  ,j_t \} \cap \{i, j \} \ne \emptyset \}~.$$
Note that $T_{ij}$ is $\cF_0$-measurable.  We have
$$
\EXP |T_{ij}| = \frac{2 k}{N+1}~. 
$$
Besides, from \cite[Proposition 1.1]{MR2288072}, for any $u >0$,
$$\PROB \left( |T_{ij}| \geq   \EXP |T_{ij}| + u \right) \leq \exp \left( - \frac{u^2  }{ 4\EXP |T_{ij}|  +2 u }  \right)~.
$$ 
If $k \leq N^{5/3} L^{-c_2}$, it follows that with overwhelming probability, the following event, say $\mathcal T$, holds: $\max_{ij} |T_{ij}| \leq 4k' /N$  where for ease of notation  we have set 
$$
k' = \min ( k , N (\log N)^2)~. 
$$

Now, let $c$ be as in Lemma \ref{loclaw} and, for $0 \leq t \leq k$, we denote by $\cE_t \in \cF_t$ the event that $\mathcal T$ holds and that the conclusion of Lemma \ref{loclaw} holds for $X^{[t]}$ and $R^{[t]}$ (with the convention $X^{[0]} = X$).  If $\cE_t$ holds, then for all  $z = E + {\mathbf{i}} \eta$ with $|2 \sqrt N - E| \leq  L^{c _0} N^{-1/6}$ and  $  \eta  =  N^{-1/6} L^{-c_1}$,  we have, 
$$
\max_{i \ne j} |R^{[t]}_{ij} (z) | \leq  \delta = L ^ {c'} N^{-5/6} \quad \hbox{ and } \quad ~ \max_{i} |R^{[t]}_{ii} (z) | \leq \delta_0 = c N^{-1/2}~,
$$
where $c' =1 + c  + \max ( c_0 / 2, c_0/4 + c_1/2) $. 

{ After these preliminaries, we may now write the resolvent expansion. Our goal is to write $R^{[k]}_{ij}(z) - R_{ij}(z)$ as a sum of martingale differences up to error terms. The outcome will be Equation \eqref{eq:RkRij} below.}
We define $X_0^{[t]}$ as the symmetric matrix obtained from $X^{[t]}$ by setting to $0$ the entries $(i_t,j_t)$ and $(j_t,i_t)$. By construction $X^{[t+1]}_0$ is $\cF_t$-measurable. We denote by $R_0^{[t]}$ the resolvent of $X_0^{[t]}$. The resolvent identity \eqref{eq:resid} implies that 
$$
R_0^{[t+1]} = R^{[t]}  + R^{[t]} ( X^{[t]}-X_0^{[t+1]})R^{[t]}     +  R^{[t]}  (X^{[t]}-X_0^{[t+1]})R^{[t]} (X^{[t]}-X_0^{[t+1]})  R^{[t+1]}_0
$$
(we omit to write the parameter $z$ for ease of notation). Now,  we set for $i \ne j$, $E^s_{ij} = e_i e_j ^* + e_j e_i^*$ and {$E_{ii}^s = e_i e_i^*$}, where $e_i$ denotes the canonical vector of $\mathbb R^n$ with all entries equal to $0$ except the $i$-th entry equal to $1$. We have 
\begin{equation}\label{eq:XtXt0}
X^{[t]}-X_0^{[t+1]} =  X_{i_{t+1} j_{t+1}} E^s_{i_{t+1} j_{t+1}} \quad  \hbox{ and } \quad X^{[t+1]}-X_0^{[t+1]} =  X'_{i_{t+1} j_{t+1}} E^s_{i_{t+1} j_{t+1}}~.
\end{equation}
We use that $|X_{ij}| \leq L/4$ and $(R^{[t+1]}_0)_{ij} \leq \eta^{-1}$. If $\cE_t$ holds, we deduce that for all $z = E + {\mathbf{i}} \eta$ with $|2 \sqrt N - E| \leq  L^{c_0} N^{-1/6}$ and  $  \eta  =  N^{-1/6} L^{-c_1}$, we have
\begin{equation}\label{eq:R0Rt}
\max_{i\ne j} |(R^{[t+1]}_0)_{ij}  | \leq  \sqrt 2 \delta \quad  \hbox{ and } \quad  \max_{i} |(R^{[t+1]}_0)_{ii}  | \leq  \sqrt 2\delta_0~. 
\end{equation}
Similarly, the resolvent identity  \eqref{eq:resid}  with $R^{[t+1]}$ and $R^{[t]}$ implies that, if $\cE_t$ holds,  for all $z = E + {\mathbf{i}} \eta$ with $|2 \sqrt N - E| \leq  L^{c_0} N^{-1/6}$ and  $  \eta  =  N^{-1/6} L^{-c_1}$, we have
\begin{equation}\label{eq:R1Rt}
\max_{i \ne j} |R^{[t+1]}_{ij}  | \leq  \sqrt 2 \delta \quad  \hbox{ and } \quad  \max_{i} |R^{[t+1]}_{ii} | \leq \sqrt  2 \delta_0~.
\end{equation}
Finally, the resolvent identity  with $R^{[t+1]}$ and $R_0^{[t+1]}$ gives
\begin{align*}
& R^{[t+1]}  \; = \;  R_0^{[t+1]} +  R_0^{[t+1]} ( X_0^{[t+1]}-X^{[t+1]})R^{[t+1]}   \\
&\quad =  \; \sum_{\ell = 0}^2 \left(R_0^{[t+1]} ( X_0^{[t+1]}-X^{[t+1]}) \right)^{\ell} R_0^{[t+1]}     +   \left(R_0^{[t+1]} ( X_0^{[t+1]}-X^{[t+1]}) \right)^{3} R^{[t+1]} ~.
\end{align*}
Note that, $\EXP [X'_{i_{t+1} j_{t+1}} | \cF_t] = 0$. We use $|X_{i_tj_t} |\leq L/4$, from \eqref{eq:XtXt0}-\eqref{eq:R0Rt}-\eqref{eq:R1Rt}, we deduce that
\begin{equation}\label{eq:ERtR0}
\left| \EXP [ R^{[t+1]}_{ij}  |\cF_t ]  - (R^{[t+1]}_0)_{ij} -s^{[t+1]}_{ij} {X'}^2_{i_{t+1}j_{t+1}} \right| \leq  a_t \quad  \hbox{ and } \quad   |R^{[t+1]}_{ij}  - (R^{[t+1]}_0)_{ij}  | \leq  b_t~,
\end{equation}
where $s^{[t]}_{ij} =  ((R_0^{[t]} E^s_{i_{t} j_{t}})^2 R_0^{[t]})_{ij}$ and, if $\cE_t$ holds, 
\begin{align*}
a_t = L^3 \delta^2 \delta_0 ^2+ L^3 \delta_0^4 \ind_{(t \in T_{ij})} \quad \hbox{ and } \quad
b_t = L \delta^2 + L \delta \delta_0 \ind_{(t \in T_{ij})} +  L \delta_0^2 \ind_{ ( \{i_t,j_t\} = \{ i,j \})}~.
\end{align*}
We rewrite, one last time the resolvent identity with  $R_0^{[t+1]}$ and $R^{[t]}$:
$$
R^{[t]} =   \sum_{\ell = 0}^2 \left(R_0^{[t+1]} ( X_0^{[t+1]}-X^{[t]}) \right)^{\ell} R_0^{[t+1]}     +   \left(R_0^{[t+1]} ( X_0^{[t+1]}-X^{[t]}) \right)^{3}R^{[t]} ~.
$$
If $\cE_t$ holds, we arrive at,  
$$
\left| \EXP [ R^{[t+1]}_{ij}  |\cF_t ]  - (R^{[t]})_{ij}  -  r^{[t+1]}_{ij} X_{i_{t+1} j_{t+1}} +  s^{[t+1]}_{ij} (X^2_{i_{t+1} j_{t+1}} - {X'}^2_{i_{t+1}j_{t+1} })   \right| \leq  2a_t~, 
$$
where $r^{[t]}_{ij} =  (R_0^{[t]} E^s_{i_{t} j_{t}} R_0^{[t]})_{ij}$. We have thus found that 
\begin{equation}\label{eq:RkRij}
R^{[k]}_{ij} - R_{ij} = \sum_{t=0}^ {k-1} \left(R^{[t+1]}_{ij}  - (R^{[t]})_{ij} \right) =  \sum_{t=0}^ {k-1} \left(R^{[t+1]}_{ij}  -  \EXP [ R^{[t+1]}_{ij}  |\cF_t ] \right) + r_{ij} + s_{ij} - s'_{ij} +a_{ij} ~, 
\end{equation}
where we have set, with $Y_{ij} = X_{ij}^2 - \EXP X_{ij}^2$, $Y'_{ij} = {X'_{ij}}^2 - \EXP X_{ij}^2$, 
$$
r_{ij} = \sum_{t=1}^{k} r^{[t]}_{ij} X_{i_{t} j_{t}}  ~,\quad s_{ij} = \sum_{t=1}^{k} s^{[t]}_{ij} Y_{i_{t} j_{t}}  ~, \quad s'_{ij} =  \sum_{t=1}^{k} s^{[t]}_{ij} Y'_{i_{t} j_{t}}  ~, \quad  |a_{ij}| \leq 2 \sum_{t=1}^k a_t~.
$$

{ In this final step of the proof, we use concentration inequalities to estimate the terms in \eqref{eq:RkRij}}. We set $Z_{t+1} =  (R^{[t+1]}_{ij}  -  \EXP [ R^{[t+1]}_{ij}  |\cF_t ] ) \IND_{\cE_t}$. We write, for any $u \geq 0$,
$$
\PROB \left( \left| \sum_{t=0}^ {k-1} \left(R^{[t+1]}_{ij}  -  \EXP [ R^{[t+1]}_{ij}  |\cF_t ] \right)\right| \geq u  \right) \leq \PROB \left( \left| \sum_{t=1}^ {k} Z_t \right| \geq u  \right) +  \sum_{t=0}^{k-1} \PROB ( \cE_t^c)~.
$$
By Lemma \ref{loclaw}, we have for any $c>0$, $\sum_{t=0}^{k-1} \PROB ( \cE_t^c) = O ( N ^{-c})$. Since $ \cE_t \in \cF_t$, we have { that} $\EXP[ Z_{t+1} | \cF_t ] = 0$. Also, from \eqref{eq:R0Rt}-\eqref{eq:ERtR0}, $|Z_{t}| \leq 2b_t$. On the event $\mathcal T$, we have 
$$
 \sqrt{\sum_{t=1}^k b_t^2 } \leq    L \delta^2 \sqrt k + L \delta \delta_0 \sqrt{\frac{4k'}{N} }  + L \delta^2_0 \leq 2 L \delta^2 \sqrt{k'}~.
$$
Azuma-Hoeffding martingale inequality implies that, for $u \geq 0$, 
$$
\PROB \left( \left| \sum_{t=1}^ {k} Z_t \right| \geq  2 u  L \delta^2 \sqrt{k'} \right) \leq 2 \exp \left( - \frac{u^2}{2}\right)~. 
$$
We apply the later inequality to $u = \log N$. We deduce that, with overwhelming probability, 
\begin{equation}\label{eq:jook}
\sum_{t=0}^ {k-1} \left(R^{[t+1]}_{ij}  -  \EXP [ R^{[t+1]}_{ij}  |\cF_t ] \right) \leq  L^2 \sqrt{k'} \delta^2~.
\end{equation}

We may treat similarly the random variable $s'_{ij}$ in \eqref{eq:RkRij}. We set $Z'_{t+1} =   s^{[t+1]}_{ij} Y'_{i_{t+1} j_{t+1}}\IND_{\cE_t}$. Note that $s^{[t+1]}_{ij}$ is $\cF_t$-measurable and $\EXP [ Y'_{i_{t+1} j_{t+1}} | \cF_t ] = 0$.  Thus $\EXP [Z'_{t+1} | \cF_t] = 0$. Moreover, since $|Y_{ij}| \leq L^2/ 16$, from \eqref{eq:R0Rt}, we find $|Z'_{t+1}| \leq b'_t =  L^2 (\delta^2 \delta_0 + \delta_0^3 \IND_{( t \in T_{ij})})$. If $\mathcal T$ holds, we get
$$
\sqrt{\sum_{t=0}^{k-1} {b'_t}^2} \leq L^2 \delta^2 \delta_0  \sqrt{k} + L^2 \delta_0^3 \sqrt{ \frac{4k'}{N}} \leq  2 L^2 \delta^2 \delta_0  \sqrt{k'} 
$$
We write, for $u \geq 0$,
$$
\PROB \left(  \left| s'_{ij} \right| \geq u \right) \leq \PROB \left(  \left| \sum_{t=1}^{k} Z'_t \right| \geq u \right)  + \sum_{t=0}^{k-1} \PROB ( \cE_{t}^c) . 
$$
From Azuma-Hoeffding martingale inequality, we deduce that, with overwhelming probability, 
\begin{equation}\label{eq:jook3}
|s'_{ij}| \leq  \sqrt{k'} \delta^2.
\end{equation}

We now estimate the random variable $r_{ij}$ in \eqref{eq:RkRij}. We will also use Azuma-Hoeffding inequality but we need to introduce a backward filtration {(because we have to deal with the random variables $X_{i_t,j_t}$ instead of $X'_{i_t,j_t}$ as in $s'_{ij}$)}. We define $\cF'_t$ as the $\sigma$-algebra generated by the random variables, $X'$, $S_k$ and $\{ (X_{ij}) : \{i,j\} \notin \{ i_s, j_s \}, s \leq t \}$. By construction $X^{[t]}$ and $X_0^{[t]}$   are $\cF'_{t}$-measurable random variables. Let $\cE'_{t} \in \cF'_{t} $ be the event that $\mathcal T$ holds and that the conclusion of Lemma \ref{loclaw} holds for $X^{[t]}$.  If $\cE'_{t}$ holds, then for all  $z = E + {\mathbf{i}} \eta$ with $|2 \sqrt N - E| \leq  L^{c_0} N^{-1/6}$ and  $  \eta  =  N^{-1/6} L^{-c_1}$,  we have, 
$$
\max_{i \ne j} |R^{[t]}_{ij} (z) | \leq  \delta \quad \hbox{ and } \quad ~ \max_{i} |R^{[t]}_{ii} (z) | \leq \delta_0~.
$$
Arguing as in \eqref{eq:R0Rt}, if $\cE'_t$ holds then 
$$
\max_{i \ne j} |(R_0^{[t]})_{ij}  | \leq  \sqrt 2\delta \quad \hbox{ and } \quad ~ \max_{i} |(R^{[t]}_0)_{ii}  | \leq \sqrt 2\delta_0~.
$$
The variable $r^{[t]}_{ij}$ is $\cF'_{t}$-measurable and $\EXP (X_{i_{t} j_{t}} | \cF'_{t} ) = 0$.  We write, for $u \geq 0$,
$$
\PROB \left(  \left| r_{ij} \right| \geq u \right) \leq \PROB \left(  \left| \sum_{t=0}^{k-1} \tilde Z_t \right| \geq u \right)  + \sum_{t=0}^{k-1} \PROB ( {\cE'_{t}}^c)~ ,
$$
where $\tilde Z_{t+1} = r^{[t]}_{ij} X_{i_{t} j_{t}}  \IND_{\cE'_{t}}$. We have $\EXP ( \tilde Z_{t+1} | \cF'_t ) = 0$ and $$|\tilde Z_t| \leq \tilde b_t = L \delta^2 +   L \delta \delta_0 \ind_{(t \in T_{ij})} +  L \delta_0^2 \ind_{ \{ i_t,j_t\} =  \{ i,j \})}~.$$ 
Arguing as above, from Azuma-Hoeffding martingale inequality, we deduce that with overwhelming probability, 
\begin{equation}\label{eq:jook2}
|r_{ij}|  \leq L^2\sqrt  {k'} \delta^2.
\end{equation}

Similarly, repeating the argument leading to \eqref{eq:jook3} with $s_{ij}$ and the filtration $(\cF'_t)$ gives with overwhelming probability, 
\begin{equation}\label{eq:jook4}
|s_{ij}| \leq  \sqrt{k'} \delta^2.
\end{equation}

We note also that if $\mathcal T$ holds then 
$$
|a_{ij}| \leq 2 \sum_{t=1}^k a_t \leq 2 L^3 \delta^2 \delta^2_0 k + 2L^3 \delta_0^4 \frac{4k'}{N} \leq \sqrt{k'} \delta^2,
$$
where the last inequality holds provided that $k \leq N^{5/3}$. So finally, from \eqref{eq:RkRij}-\eqref{eq:jook}-\eqref{eq:jook3}-\eqref{eq:jook2}-\eqref{eq:jook4}, we have proved that for a given $z = E + {\mathbf{i}} \eta$  such that $|E - 2 \sqrt N| \leq L^{c_0} N^{-1/6}$ and $\eta = N^{-1/6} L ^{-c_1}$, with overwhelming probability
$$
\left| R^{[k]}_{ij} (z)- R_{ij} (z) \right| \leq 3 L ^2  \sqrt {k'} \delta^2~,
$$
where the  inequality holds provided that $k \leq N^{5/3}$. 
Recall that $|R_{ij} (E + {\mathbf{i}} \eta) - R_{ij} ( E' + {\mathbf{i}} \eta) | \leq \eta^{-2} |E - E'|$. By a net argument (as in the proof of Lemma \ref{lem:resres0}), we deduce that with overwhelming probability for all $z = E + {\mathbf{i}} \eta$  such that $|2 \sqrt N - E| \leq  L^{c_0} N^{-1/6}$, $\left| R^{[k]}_{ij} (z)- R_{ij} (z) \right| \leq 4 L^2 \sqrt {k'} \delta^2$.  It concludes the proof of Lemma  \ref{lem:resresk}.

\subsection{Proof of Lemma \ref{lem:llk}}

Let $c_0$ be as in Lemma \ref{lem:resl} and $c >0$. We set $c_1 = c_0/2 +2c$ and let $\eta = N^{-1/6} L ^{-c_1}$. Let $p \in \{1,2\}$. We start with by bounding $\min_{ j}|\lambda_p - \lambda_j^{[k]}|$ and $\min_{j}|\lambda^{[k]}_p - \lambda_j|$. Since $X$ and $X^{[k]}$ have the same distribution, we only prove that with overwhelming probability 
\begin{equation}\label{eq:pton}
\min_{ 1 \leq j \leq N}|\lambda_p - \lambda_j^{[k]}| \leq 2 L^{c_0/2} \eta.
\end{equation}
By Lemma \ref{lem:resl}, with overwhelming probability,  $|\lambda_p - 2 \sqrt N | \leq L^{c_0} N^{-1/6}$ and for some integer $1 \leq i\leq N$, 
$$
N \eta^{-1} \Im R (\lambda_p + {\mathbf{i}} \eta)_{ii} \geq \frac 1 2 \eta^{-2}~, 
$$
and, 
$$
N \eta^{-1} \Im R^{[k]} (\lambda_p + {\mathbf{i}} \eta)_{ii} \leq L^{c_0} \min_{1 \leq j \leq N} (\lambda_p - \lambda_j^{[k]})^{-2}~. 
$$
By Lemma \ref{lem:resresk}, we deduce that if $k \leq N^{5/3} L^{-c_2}$, with overwhelming probability,
$$
\frac 1 4 \eta^{-2} \leq L^{c_0}\min_{1 \leq j \leq N} (\lambda_p - \lambda_j^{[k]})^{-2}~. 
$$
It proves  \eqref{eq:pton}.

We may now conclude the proof of Lemma \ref{lem:llk}. Fix $\eps >0$. As already noticed, from  \cite[Theorem 2.7]{MR3253704}, there exists $\delta >0$ such that, with probability at least $1 - \eps$, $\lambda_2 < \lambda - \delta N^{-1/6}$. From what precedes, with probability at least $1 - 2 \eps$, $\cE_\delta$ holds and for all $k \leq N^{5/3} L^{-c_2}$, we have 
$$
\max \left( \min_{ 1 \leq j \leq N}|\lambda^{[k]}_p - \lambda_j|  , \min_{ 1 \leq j \leq N}|\lambda_p - \lambda_j^{[k]}| \right) \leq \alpha~,
$$
with $\alpha = 2 L^{c_0/2} \eta$. 
On this event, we readily find $|\lambda - \lambda^{[k]} | \leq \alpha$ and for some $p$,   $|\lambda_p - \lambda^{[k]}_2| \leq \alpha$. Assume that this last inequality is false for $p \ne 2$. Since $2\alpha < \delta N^{-1/6}$, if $p \ne 2$, then $p \leq 3$ and we deduce that $\lambda_2 > \lambda_2^{[k]} + \alpha$. We note that, on our event, for some $q$, we have $|\lambda_2 - \lambda_q^{[k]}| \leq \alpha$. In particular, $\lambda_q^{[k]} > \lambda_2^{[k]}$. So necessarily, $q = 1$ and, from the triangle inequality, $| \lambda_2 - \lambda_1| \leq 2 \alpha$. This is a contradiction since $2\alpha < \delta N^{-1/6}$. It  concludes the proof of Lemma \ref{lem:llk}.

{\textbf{Acknowledgments} We would like to thank Jaehun Lee for pointing out a mistake in the proof of Lemma \ref{lem:monotonicity_second} in an early version of this paper. We also would like to thank the referees for their valuable reports.}

\bibliographystyle{plain}
\bibliography{noisesensitivity}

\end{document}